\documentclass[12pt,reqno]{amsart}
\usepackage{fullpage}
\usepackage[english]{babel} 
\usepackage[utf8]{inputenc}  
\usepackage{amsmath, amsfonts, amssymb}
\usepackage{hyperref, comment}
\usepackage{color}

\theoremstyle{plain}
\newtheorem{theo}{Theorem}[section]
\newtheorem*{theo*}{Theorem}

\newtheorem{lemma}[theo]%
 {Lemma}
 {Lemma}
\newtheorem{prop}[theo]%
{Proposition}
\newtheorem{coro}[theo]%
{Corollary}

\theoremstyle{definition}
{Definition}
{Definition-Proposition}
{Conjecture} 

\theoremstyle{remark}
\newtheorem{remark}[theo]%
{Remark}
{Example}
\newtheorem{question}[theo]
{Question}

\makeatletter
\newcounter{enum@ux}
\def\brkenum#1{%
\setcounter{enum@ux}{\value{enum\romannumeral\the\@enumdepth}}%
\end{enumerate} #1%
\begin{enumerate}%
\addtocounter{enum\romannumeral\the\@enumdepth}{\value{enum@ux}}}
\makeatother

\makeindex
\newcommand{\new}{\mathrm{new}}
\newcommand{\Vol}{\mathrm{Vol}}
\newcommand{\nor}{\mathrm{nor}}

\newcommand{\ds}{\displaystyle}
\newcommand{\Avg}{\mathop\mathrm{Avg}}
\newcommand{\Avgh}{\mathop{\mathrm{Avg}^h}\limits_{f\in B_k(N)}}
\newcommand{\Sym}{\mathop\mathrm{Sym}}

\renewcommand{\mod}{\mathop\mathrm{mod}}
\renewcommand {\Re}{\mathop{\mathrm{Re}}}
\renewcommand {\Im}{\mathop{\mathrm{Im}}}

\newcommand{\fg}{\mathfrak{f}}

\newcommand{\qg}{\mathfrak{q}}

\newcommand{\Lg}{\mathfrak{L}}
\newcommand{\lgo}{\mathfrak{l}}
\newcommand{\hg}{\mathfrak{h}}

\newcommand{\ggo}{\mathfrak{g}}

\newcommand{\bbN}{\mathbb{N}}
\newcommand{\bbR}{\mathbb{R}}
\newcommand{\bbF}{\mathbb{F}}
\newcommand{\Q}{\mathbb{Q}}
\newcommand{\Z}{\mathbb{Z}}
\newcommand{\C}{\mathbb{C}}

\newcommand{\calC}{\mathcal{C}}
\newcommand{\calD}{\mathcal{D}}
\newcommand{\calF}{\mathcal{F}}
\newcommand{\calH}{\mathcal{H}}

\newcommand{\calM}{\mathcal{M}}
\newcommand{\calP}{\mathcal{P}}

\begin{document}

\title[]{On $M$-functions associated with modular forms}

\author{Philippe Lebacque}
\address{
Philippe Lebacque
\newline \indent
Laboratoire de Math\'ematiques de Besan\c con, UFR Sciences et techniques 
\newline \indent
16, route de Gray 25 030 Besan\c con, France
}
\email{philippe.lebacque@univ-fcomte.fr}

\author{Alexey Zykin}
\address{
Alexey Zykin
\newline \indent
Laboratoire GAATI, Universit\'e de la Polyn\'esie fran\c caise
\newline \indent
BP 6570 --- 98702 Faa'a, Tahiti, Polyn\'esie fran\c caise
\newline \indent
National Research University Higher School of Economics
\newline \indent
AG Laboratory NRU HSE 
\newline \indent
Institute for Information Transmission Problems of the Russian Academy of Sciences
}
\email{alzykin@gmail.com}

\thanks{The two authors were partially supported by ANR Globes ANR-12-JS01-0007-01 and the second author by the Russian Academic Excellence Project '5-100'. }

\date{}

\begin{abstract}
Let $f$ be a primitive cusp form of weight $k$ and level $N,$ let $\chi$ be a Dirichlet character of conductor coprime with $N,$ and let $\Lg(f\otimes \chi, s)$ denote either $\log L(f\otimes \chi, s)$ or $(L'/L)(f\otimes \chi, s).$  In this article we study the distribution of the values of $\Lg$ when either $\chi$ or $f$ vary. First, for a quasi-character $\psi\colon \C \to \C^\times$ we find the limit for the average $\Avg_\chi \psi(L(f\otimes\chi, s)),$ when $f$ is fixed and $\chi$ varies through the set of characters with prime conductor that tends to infinity. Second, we prove an equidistribution result for the values of $\Lg(f\otimes \chi,s)$ by establishing analytic properties of the above limit function. Third, we study the limit of the harmonic average $\Avg^h_f \psi(L(f, s)),$ when $f$ runs through the set of primitive cusp forms of given weight $k$ and level $N\to \infty.$ Most of the results are obtained conditionally on the Generalized Riemann Hypothesis for $L(f\otimes\chi, s).$ 
\end{abstract}



\maketitle

\section{Introduction}

\subsection{Some history}
The study of the distribution of values of $L$-functions is a classical topic in number theory.  In the first half of 20th century Bohr, Jessen, Wintner, etc. intiated a study of the distribution of the values of the logarithm $\log \zeta(s)$ and the logarithmic derivative $(\zeta'/\zeta)(s)$ of the Riemann zeta-function, when $\Re s=\sigma > \frac{1}{2}$ is fixed and $\Im s=\tau\in\bbR$ varies (\cite{BJ}, \cite{BorchJ},\cite{JW}, \cite{KW}). This was later generalized to $L$-functions of cusp forms and Dedekind zeta-functions by Matsumoto (\cite{M1}, \cite{M2}, \cite{M3}).

In the last decade Y. Ihara in \cite{Ih1} proposed a novel view on the problem by studying other families of $L$-functions. His initial motivation was to investigate the properties of the Euler--Kronecker constant $\gamma_K$ of a global field $K,$ which was defined by him in \cite{Ih0} to be the constant term of the Laurent series expansion of the logarithmic derivative of the Dedekind zeta function of $K,$ $\zeta'_K(s)/\zeta_K(s).$ The study of $L'(1,\chi)/L(1,\chi)$ initiated in \cite{IMS} grew out to give a whole range of beautiful results on the value distribution of $L'/L$ and $\log L$. 

Given a global field $K$, i.e. a finite extension of $\Q$ or of $\bbF_q(t),$ and a family of characters $\chi$ of $K$ Ihara considered in \cite{Ih1} the distribution of $L'(s, \chi)/L(s, \chi)$ in the following cases:

\begin{itemize}
\item[(A)] $K$ is $\Q,$ a quadratic extension of $\Q$ or a function field over $\bbF_q,$ $\chi$ are Dirichlet characters on $K;$
\item[(B)] $K$ is a number field with at least two archimedean primes, and $\chi$ are normalized unramified Grössencharacters;
\item[(C)] $K=\Q$ and $\chi=\chi_{t}, t\in\bbR$ defined by $\chi_t(p)=p^{-it}.$ 
\end{itemize}

The equidistribution results of the type
\begin{equation}
\label{AverageIh}
\Avg\nolimits'_{\chi} \Phi\left(\frac{L'(s, \chi)}{L(s,\chi)}\right)=\int_\C M_\sigma (w) \Phi(w)|dw|,
\end{equation}
(with a suitably defined average in each of the above cases) were proven for $\sigma=\Re s > 1$ for number fields, and for $\sigma > 3/4$ for function fields, under significant restrictions on the test function $\Phi$. The function field case was treated once again in \cite{IM1} by Y. Ihara and K. Matsumoto, with both the assumptions on $\Phi$ and on $\sigma$ having been relaxed ($\Phi$ of at most polynomial growth and $\sigma > 1/2$ respectively). The most general results in the direction of the case (A) were established in \cite{IM3} conditionally under the Generalized Riemann Hypothesis (GRH) in the number field case and unconditionally in the function field case (the Weil's Riemann hypothesis being valid) for both families $L'(s, \chi)/L(s, \chi)$ and $\log L(s, \chi).$ For $\Re s>\frac{1}{2}$ Ihara and Matsumoto prove that
$$
\Avg\nolimits_{\chi} \Phi\left(\frac{L'(s, \chi)}{L(s,\chi)}\right)=\int_\C M_\sigma (w) \Phi(w)|dw|, \quad \Avg\nolimits_{\chi} \Phi(\log L(s, \chi))=\int_\C \calM_\sigma (w) \Phi(w)|dw|,
$$
for continuous test functions $\Phi$ of at most exponential growth. Note that $\Avg'$ in \eqref{AverageIh} is different from the one used in the latter paper, since extra averaging over conductors is assumed in the former case, the resulting statements being weaker.

Unconditional results for a more restrictive class of $\Phi$ (bounded continous functions), and with extra averaging over the conductor $\Avg',$ but still for $\Re s > \frac{1}{2}$ were established in \cite{IM2} and \cite{IM4} in the $\log$ and $\log'$ cases respectively in the situations (A, $K=\Q$) and (C).

The above results give rise to the density functions $M_\sigma(z)$ and a related function $\tilde{M}_s(z_1, z_2)$ (which is the inverse Fourier transform of $M_\sigma,$ when $z_2=\bar{z}_1$, $s=\sigma\in \bbR$) both in the $\log$ and $\log'$ cases. Under optimal circumstances (though it is very far from being known unconditionally in all cases) we have
$$
M_\sigma(z)=\Avg\nolimits_\chi \delta_z\left(\Lg(\chi, s)\right), \quad \tilde{M}_\sigma(z_1, z_2)=\Avg\nolimits_\chi \psi_{z_1, z_2}\left(\Lg(\chi, s)\right),
$$
where $\Lg(s,\chi)$ is either $L'(s, \chi)/L(s,\chi)$ or $\log L(s,\chi),$ $\delta_z$ is the Dirac delta function, and $\psi_{z_1, z_2}(w)=\exp\left(\frac{i}{2}(z_1\bar{w}+z_2w)\right)$ is a quasi-character.

The functions $M$ and $\tilde{M}$ turn out to have some remarkable properties that can be established unconditionally. For example, $\tilde{M}$ has an Euler product expansion, an analytic continuation to the left of $\Re s> 1/2,$ its zeroes and the ``Plancherel volume'' $\int_\C |\tilde{M}_\sigma(z,\bar{z})|^2|dz|$  are interesting objects to investigate. We refer to \cite{Ih2}, \cite{Ih3} for an in-depth study of $M$ and $\tilde{M},$ as well as to the survey \cite{IM0} for a thorough discussion of the above topics. 

In a recent paper by M. Mourtada and K. Murty \cite{MM} averages over quadratic characters were considered. Using the methods from \cite{IM3}, they establish an equidistribution result conditional on GRH. Note that in their case the values taken by the $L$-functions are real. In this respect the situation is similar to the one considered by us in \S \ref{section_primitive} in case we assume that $s$ is real.

Finally, let us quote a still more recent preprint by K. Matsumoto and Y. Umegaki \cite{MU} that treats similar questions for differences of logarithms of two symmetric power $L$-functions under the assumption of the GRH. Their approach is based on \cite{IM2} rather than on \cite{IM3}, though the employed techniques are remarkably close to the ones we apply in \S\ref{section_primitive}. The results of Matsumoto and Umegaki are complementary to ours, since the case of $\Sym^1 f = f,$ which is the main subject of our paper, could not be treated in \cite{MU}.

\subsection{Main results}

In this article, we generalize to the case of modular forms the methods of Ihara and Matsumoto to understand the average values of $L$-functions of Dirichlet characters over global fields. 

Our results are obtained in two different setting. First, we consider the case of a fixed modular form, while averaging with respect to its twists by Dirichlet characters. Our results in this setting are fairly complete, though sometimes conditional on GRH. Second, we consider averages with respect to primitive forms of given weight and level, when the level goes to infinity. 

Let us formulate our main results. A more thorough presentation of the corresponding notation can found in \S \ref{section_notation} and in the corresponding sections.

Let $B_k(N)$ denote the set of primitive cusp forms of weight $k$ and level $N,$ let $f\in B_k(N),$ and let $\chi$ be a Dirichlet character of conductor $m$ coprime with $N.$ Define $\Lg(f\otimes \chi, s)$ to be either $(L'/L)(f\otimes\chi,s)$ or $\log L(f\otimes \chi, s),$ put $\ggo(f\otimes \chi, s, z)=\exp\left(\frac{iz}{2}\Lg(f, s)\right).$ We introduce $\lgo_z(n)$ to be the coefficients of the Dirichlet series expansion $\ggo(f\otimes \chi, s, z)=\sum\limits_{n\geq 1} \lgo_z(n) n^{-s}.$ Using the relations between the coefficients of the Dirichlet series expansion $L(f, s)=\sum\limits_{n\geq 1} \eta_f(n) n^{-s},$ one can write $\lgo_z(n)=\sum\limits_{x\geq 1} c_{z, x}^N(n) \eta_f(x),$ where $c_{z, x}^N(n)$ depend only on the level $N.$ Put $c_{z, x}(n)=c_{z,x}^1(n).$

In what follows, the expressions of the form $f \ll g,$ $g\gg f,$ and $f = O(g)$ all denote that $|f|\leq c|g|,$ where $c$ is a positive constant. The dependence of the constant on additional parameters will be explicitly indicated (in the form $\ll_{\epsilon, \delta, \dots}$ or $O_{\epsilon, \delta, \dots}$), if it is not stipulated otherwise in the text. We denote by $v_p(n)$ the $p$-adic valuation of $n,$ writing as well $p^k \parallel n$ if $v_p(n)=k.$ We also use the notation $:=$ or $=:$ meaning that the corresponding object to the left or to the right of the equality respectively is defined in this way.

Our main results are as follows.

\begin{theo*} [Theorem \ref{averagechar}]
Assume that $m$ is a prime number and let $\Gamma_{m}$ denote the group of Dirichlet characters modulo $m.$ Let $0<\epsilon<\frac{1}{2}$ and $T,R>0.$ Let  $s=\sigma+it$ belong to the domain $\sigma\geq\epsilon+\frac{1}{2},$ $|t|\leq T,$ let $z$ and $z'$ be inside the disk $\calD_R=\{z \mid |z| \leq R\}$. Then, assuming the Generalized Riemann Hypothesis (GRH) for $L(f\otimes\chi, s)$, we have
$$
\lim_{m\to \infty}\frac{1}{|\Gamma_m|} \sum_{\chi\in\Gamma_{m}}\overline{\ggo(f\otimes\chi,s,z)}{\ggo(f\otimes\chi,s,z')}=\sum_{n\geq 1}\overline{\lgo_{z}(n)}\lgo_{z'}(n)n^{-2\sigma}=: \tilde{M}_\sigma(-\bar{z}, z').
$$
\end{theo*}

\begin{theo*}[Theorem \ref{maindistrib}]
Let $\Re s = \sigma>\frac{1}{2}$ and let $m$ run over prime numbers. Let $\Phi$ be either a continuous function on $\C$ with at most exponential growth, or the characteristic function of a bounded subset of $\C$ or of a complement of a bounded subset of $\C.$ Define $M_\sigma$ as the inverse Fourier transform of $\tilde{M}_\sigma(z, \bar{z}).$ Then under GRH for $L(f\otimes \chi,s)$ we have 
\begin{equation*}
\lim_{m\to\infty}\frac{1}{|\Gamma_m|} \sum_{\chi\in\Gamma_{m}}\Phi(\Lg(f\otimes\chi,s))=\int_{\C}M_{\sigma}(w)\Phi(w)|dw|.
\end{equation*}
\end{theo*}

\begin{theo*}[Theorem \ref{average_primitive}] 
Assume that $N$ is a prime number and that $k$ is fixed.
Let $0<\epsilon<\frac{1}{2}$ and $T,R>0.$ Let  $s=\sigma+it$ belong to the domain $\sigma\geq\epsilon+\frac{1}{2},$ $|t|\leq T,$ and $z$ and $z'$ to the disc $\calD_R$ of radius $R.$ Then, assuming GRH for $L(f, s),$ we have 
$$
\lim_{N\to +\infty} \sum_{f\in B_k(N)} \omega(f) \overline{\ggo(f,s,z)}\ggo(f,s,z')=\sum_{n,m\in\bbN}n^{-\bar{s}}m^{-s}\sum_{x\geq 1}\overline{c_{z,x}(n)}c_{z',x}(m),
$$
where $\omega(f)$ are the harmonic weights defined in \S \ref{section_primitive}.
\end{theo*}

Finally, let us describe the structure of the paper. In \S\ref{section_notation} we introduce the notation and some technical lemmas to be used throughout the paper. The \S\ref{section_twists} is devoted to the proof of Theorem \ref{averagechar} on the mean values of the logarithms and logarithmic derivatives of $L$-functions obtained by taking averages over the twists of a given primitive modular form. Using GRH, we deduce it from Ihara and Matsumoto's results. In \S\ref{section_distribution} we study unconditionally the analytic properties of $M$ and $\tilde{M}$ functions in the above setting. We then prove an equidistribution result (Theorem \ref{maindistrib}), which is, once again, conditional on GRH. In \S\ref{section_primitive}, we consider the average over primitive forms of given weight $k$ and level $N,$ when $N\to\infty$, establishing under GRH Theorem \ref{average_primitive}. The orthogonality of characters is replaced by the Petersson formula in this case, which obviously makes the proofs trickier. Finally, open questions, remarks and further research directions are discussed in \S\ref{section_question}.

\textbf{Acknowledgements.}

We would like to thank Yasutaka Ihara for helpful discussions. The first author would like to express his gratitude to the INRIA team GRACE for an inspirational atmosphere accompanying his stay, during which a large part of this work was done.

\section{Notation}
\label{section_notation}

The goal of this section is to introduce the notation necessary to state our main results. We also prove some auxiliary estimates to be used throughout the paper. 

\subsection{The \texorpdfstring{$\ggo$}{g}-functions}

Let $N$, $k$ be two integers. We denote by $S_k(N)$ the set of cusp forms of weight $k$ and level $N,$ and by $S^{\new}_k(N)$ the set of new forms. For $f\in S_k(N)$  we write $f(z)=\sum\limits_{n=1}^{\infty} \eta_f(n)n^{(k-1)/2}e(nz)$ for its Fourier expansion at the cusp $\infty,$ with the standard notation $e(nz)=e^{2\pi inz}.$

Let $B_k(N)$ denote the set of primitive forms of weight $k$ and level $N,$ i.e. the set of $f^{\nor}=f/\eta_f(1)$ where $f$ runs through an orthogonal basis of $S_k^\new(N)$ consisting of eigenvectors of all Hecke operators $T_n$, so that the Fourier coefficients of the elements of $B_k(N)$ are the same as their Hecke eigenvalues. Note that for a primitive form $f\in B_k(N)$ all its coefficients $\eta_f(n)$ are real.

The $L$-function of a primitive form $f\in B_k(N)$ is defined as the Dirichlet series $L(f, s)=\sum\limits_{n=0}^\infty \eta_f(n)n^{-s}.$ The series converges absolutely for $\Re s> 1,$ however, $L(f,s)$ can be analytically continued to an entire function on $\C.$ It admits the Euler product expansion:
$$
L(f,s)=\prod_p L_{p}(f,s),
$$ 
where, for any prime number $p,$
$$L_p(f,s)=\begin{cases}
            \left(1-\eta_f(p)p^{-s}+p^{-2s}\right)^{-1} & \text{ if } (p,N)=1,\\
	    \left(1-\eta_f(p)p^{-s}\right)^{-1} & \text{ if } p\mid N.
           \end{cases}
$$
By the results of Deligne, these local factors can be written as follows (\cite[Chapter 6]{Iw} or \cite[Chapter IX, \S 7]{Kn}): 
% Only true for newforms, some modification is needed for old forms, since they are not eigenvectors for all T_p. However, everything is fine for $p\nmid N.$
\begin{equation}
\label{prodform}
L_p(f,s)=\left(1-\alpha_f(p)p^{-s}\right)^{-1}\left(1-\beta_f(p)p^{-s}\right)^{-1},
\end{equation}
where 
$$
  \begin{cases} 
 |\alpha_f(p)|=1,\ \beta_f(p)= \alpha_f(p)^{-1} & \text{ if } (p,N)=1,\\
	    \alpha_f(p)=\pm p^{-{\frac{1}{2}}}, \beta_f(p)=0& \text{ if } p\parallel N\ (\text{that is } p\mid N \text{ and } p^2\nmid N),\\
			\alpha_f(p)=\beta_f(p)=0& \text{ if } p^2 \mid N.
   \end{cases}
$$

We are interested in the two functions
$$
\begin{cases}
\ds g(f,s,z)=\exp\left(\frac{iz}{2} \frac{L'(f,s)}{L(f,s)}\right),\\ 
\ds G(f,s,z)=\exp\left(\frac{iz}{2}\log L(f,s)\right).
\end{cases}
$$
Define $h_n(x)$ and $H_n(x)$ as the coefficients of the following generating functions:
\begin{align*}
\exp\left(\frac{xt}{1-t}\right)&=\sum_{n=0}^{+\infty} h_n(x) t^n,\\
\exp(-x\log(1-t))&=\sum_{n=0}^{+\infty} H_n(x) t^n,
\end{align*}
or, equivalently (cf. \cite[\S 1.2]{IM3}), as the functions given by  $h_0(x)=H_{0}(x)=1$ and, for $n\geq 1,$ 
\begin{align*}
h_n(x)&=\sum_{r=0}^n\frac{1}{r!}\binom{n-1}{r-1}x^r,\\
H_n(x)&=\frac{1}{n!} x(x+1)\dots(x+n-1).
\end{align*}

As we have $$\frac{iz}{2} \frac{L'(f,s)}{L(f,s)}=-\frac{iz}{2}\sum_p \frac{\alpha_f(p)p^{-s}\log p }{1-\alpha_f(p)p^{-s}}+\frac{\beta_f(p)p^{-s}\log p }{1-\beta_f(p)p^{-s}},$$ we can write (using the standard convention that, in the case when $\beta_f(p)=0,$ we put $\beta_f(p)^n = 0,$ if $n>0,$ and $\beta_f(p)^0=1$):
\begin{align*}
g(f,s,z)&=\exp\left(\frac{iz}{2} \frac{L'(f,s)}{L(f,s)}\right)=\\
&=\prod\limits_{p} \exp\left(\frac{\alpha_f(p)p^{-s}}{1-\alpha_f(p)p^{-s}}\cdot \frac{-iz\log p}{2}\right)\exp\left(\frac{\beta_f(p)p^{-s}}{1-\beta_f(p)p^{-s}}\cdot \frac{-iz\log p}{2}\right)=\\
&=\prod\limits_{p} \left(\sum\limits_n h_n\left(-\frac{iz}{2}\log p\right)\alpha_f(p)^n p^{-ns}\right)\left(\sum\limits_n h_n\left(-\frac{iz}{2}\log p\right)\beta_f(p)^n p^{-ns}\right)=\\
&=\prod_{p}\left( \sum_{n=0}^{+\infty}\sum_{r=0}^n h_r\left(-\frac{iz}{2}\log p\right)h_{{n-r}}\left(-\frac{iz}{2}\log p\right)\alpha_f(p)^{r}\beta_f(p)^{n-r}p^{-ns}\right) =\\
&=\prod_{p\nmid N}\left( \sum_{n=0}^{+\infty}\sum_{r=0}^n h_r\left(-\frac{iz}{2}\log p\right)h_{{n-r}}\left(-\frac{iz}{2}\log p\right)\alpha_f(p)^{2r-n} p^{-ns}\right)
\cdot\\
&\cdot\prod_{p\parallel N}\sum_{n=0}^{+\infty}h_n\left(-\frac{iz}{2}\log p\right)\alpha_f(p)^{n}p^{-ns}=:\prod_p \sum_{n=0}^{+\infty}\lambda_z(p^n) p^{-ns}.
\end{align*}

In a similar way we get:

\begin{align*}
G(f,s,z)&=\exp\left(\frac{iz}{2} \log{L(f,s)}\right)=\\
&= \prod\limits_p\exp\left(-\frac{iz}{2}\log(1-\alpha_p(f)p^{-s})\right)\exp\left(-\frac{iz}{2}\log(1-\beta_p(f)p^{-s})\right)=\\ 
&=\prod_{p\nmid N}\left( \sum_{n=0}^{+\infty}\sum_{r=0}^nH_r\left(\frac{iz}{2}\right)H_{{n-r}}\left(\frac{iz}{2}\right)\alpha_f(p)^{2r-n}p^{-ns}\right)\prod_{p\parallel N} \sum_{n=0}^{+\infty}H_n\left(\frac{iz}{2}\right)\alpha_f(p)^{n}p^{-ns}\\
&=:\prod_p \sum_{n=0}^{+\infty}\Lambda_z(p^n) p^{-ns}.
\end{align*}

 We extend multiplicatively $\lambda_z$ and $\Lambda_z$ to $\bbN$ so that we can write:
\begin{align*}
g(f,s,z)&= \sum_{n\geq 1} \lambda_z(n) n^{-s},\\
G(f,s,z)&= \sum_{n\geq 1} \Lambda_z(n) n^{-s}.
\end{align*}

We will use the notation $\Lg$ for $\ds\frac{L'(f,s)}{L(f,s)}$ or $\log{L(f,s)},$ $\ggo$ for $g$ or $G,$ $\hg_z(p^n)$ for $h_n\left(-\frac{iz}{2}\log p\right)$ or $H_n\left(\frac{iz}{2}\right)$, and $\lgo$ for $\lambda$ or $\Lambda$ depending on the case we consider. Thus, we can write in a uniform way:

\begin{align*}
\ggo(f,s,z)&=\exp\left(\frac{iz}{2}\Lg(f,s)\right)= \sum_{n\geq 1} \lgo_z(n) n^{-s} =\prod_p \sum_{n=0}^{+\infty}\lgo_z(p^n) p^{-ns}=\\
  = &\prod_{p\nmid N}\left(\sum_{n=0}^{+\infty}\sum_{r=0}^n \hg_z(p^r)\hg_z(p^{n-r})\alpha_f(p)^{2r-n}p^{-ns}\right) \prod_{p\parallel N} \sum_{n=0}^{+\infty}\hg_z(p^n)\alpha_f(p)^{n}p^{-ns}.
\end{align*}

The coefficients $\lgo_z(n)$ will be used to define the $\tilde{M}$-functions in the case of averages over twists of modular forms by Dirichlet characters.

\subsection{The coefficients  \texorpdfstring{$\lgo_z(n)$}{l} and \texorpdfstring{$c_{z,x}(n)$}{c}}

In this subsection we will find a more explicit expression for $\lgo_z(n).$ For $p\nmid N$ we will use the formula (see \cite[(3.5)]{RW})
$$
\eta_f(p^r)=\frac{\alpha_f(p)^{r+1}-\beta_f(p)^{r+1}}{\alpha_f(p)-\beta_f(p)},
$$
which easily follows from \eqref{prodform}. Taking into account that $\beta_f(p)=\bar{\alpha}_f(p),$ we have for $r\geq 2$
\begin{align*}
\eta_f(p^r)&=\frac{\alpha_f(p)^{r+1}-\overline{\alpha_f(p)}^{r+1}}{\alpha_f(p)-\overline{\alpha_f(p)}}=\sum_{i=0}^{r}\alpha_f(p)^i\overline{\alpha_f(p)}^{r-i}=\sum_{i=0}^{r}\alpha_f(p)^{r-2i}=\\
&=\alpha_f(p)^r+\overline{\alpha_f(p)}^r+\sum_{i=1}^{r-1}\alpha_f(p)^{r-2i}=\alpha_f(p)^r+\overline{\alpha_f(p)}^r+\sum_{i=0}^{r-2}\alpha_f(p)^{r-2i-2}=\\
&=\alpha_f(p)^r+\overline{\alpha_f(p)}^r+\eta_f(p^{r-2}).
\end{align*}
%%%% Alternative writing of the same%%%%%%
%%%%%%%%%%%%%%%%%%%%%%%%%%%%
\begin{comment}
\begin{align*}
\eta_f(p^r)&=\frac{\alpha_f(p)^{r+1}-\bar{\alpha}_f(p)^{r+1}}{\alpha_f(p)-\bar{\alpha}_f(p)}=\sum_{i=0}^{r}\alpha_f(p)^i\bar{\alpha}_f(p)^{r-i}=\sum_{i=0}^{r}\alpha_f(p)^{r-2i}=\\
&=\alpha_f(p)^r+\bar{\alpha}_f(p)^r+\sum_{i=1}^{r-1}\alpha_f(p)^{r-2i}=\alpha_f(p)^r+\bar{\alpha}_f(p)^r+\sum_{i=0}^{r-2}\alpha_f(p)^{r-2i-2}=\\
&=\alpha_f(p)^r+\bar{\alpha}_f(p)^r+\eta_f(p^{r-2}).
\end{align*}
\end{comment}
%%%%%%%%%%%%%%%%%%%%%%%%%%%%%%%
The above formula also holds for $r= 1$ if we put $\eta_f(p^{-1})=0$. From this we deduce that
$$
\alpha_f(p)^r+\beta_f(p)^r=\eta_f(p^r)-\eta_f(p^{r-2}).
$$

Using the previous formula, we can write
 \begin{align*}
  \lgo_z(p^r)&= \sum_{a=0}^r \hg_z(p^a)\hg_z(p^{r-a})\alpha_f(p)^{2a-r}\\
 & =\hg_z(p^{\frac{r}{2}})^2+\sum_{a=0}^{\lfloor{\frac{r-1}{2}}\rfloor}\hg_z(p^a)\hg_z(p^{r-a})\left(\alpha_{f}(p)^{r-2a}+\alpha_{f}(p)^{2a-r}\right)\\
 &=\hg_z(p^{\frac{r}{2}})^2 + \sum_{a=0}^{\lfloor{\frac{r-1}{2}}\rfloor}\hg_z(p^a)\hg_z(p^{r-a})\left(\eta_{f}(p^{r-2a})-\eta_{f}(p^{r-2a-2})\right)\\
 &=\hg_z(p^\frac{r}{2})^2-\hg_z(p^{\frac{r}{2}-1})\hg_z(p^{\frac{r}{2}+1}) + \sum_{a=0}^{\lfloor{\frac{r-1}{2}}\rfloor}(\hg_z(p^a)\hg_z(p^{r-a})-\hg_z(p^{a-1})\hg_z(p^{r-a+1}))\eta_{f}(p^{r-2a})\\
&=\sum_{a=0}^{\lfloor{\frac{r}{2}}\rfloor}(\hg_z(p^a)\hg_z(p^{r-a})-\hg_z(p^{a-1})\hg_z(p^{r-a+1}))\eta_{f}(p^{r-2a}),
 \end{align*}
where we put $\hg_z(p^\frac{r}{2})=\hg_z(p^{\frac{r}{2}-1})=0,$ if $r$ is odd, and $\hg_z(p^a)=0,$ if $a<0.$ 

When $p\mid N$ we have 
$$\lgo_z(p^r)=\hg_z(p^r)\alpha_f(p)^r=\hg_z(p^r)\eta_f(p)^r=\hg_z(p^r)\eta_f(p^r).$$ 
 % when $p\parallel N$, and $\lgo_z(p^r)=0$ if $p^2\mid N.$

Denoting by $\calP$ the set of prime numbers, for $n=\prod\limits_{p\in\calP}p^{v_{p}(n)}$ put 
$$
I_N(n)=\{m\in\bbN \mid v_{p}(m) \equiv v_{p}(n) \mod 2\ \text{ for } p\in\calP, v_p(n)=v_p(m) \text{ if } p\mid N\}
$$ 
and 
$$
J_N(n)=\{m\in I_N(n) \mid v_{p}(m)\leq v_{p}(n) \text{ for all } p\in\calP \}.
$$ 
Note the following easy estimate (\cite[Theorem 315]{HW}) in which $\tau(n)$ is the number of divisors of $n$:
\begin{equation}
\label{sizeJ}
|J_N(n)|=\prod_{p\mid n}\left(\left\lfloor\frac{v_{p}(n)}{2}\right\rfloor+1\right)\leq \tau(n)\ll_\epsilon n^\epsilon.
\end{equation}

The previous computations may be summarized as follows:
$$
\lgo_z(p^r)=\sum_{x\in J_N(p^r)} c_{z,x}^N(p^r)\eta_f(x),
$$
where 
$$
c_{z, p^a}^N(p^r)=
\begin{cases}
\hg_z(p^{\frac{r-a}{2}})\hg_z(p^{\frac{r+a}{2}})-\hg_z(p^{\frac{r-a}{2}-1})\hg_z(p^{\frac{r+a}{2}+1}),& \text{ if } p\nmid N \text{ and } r\equiv a \mod 2, \\
\hg_z(p^r),& \text{ if } p\mid N \text{ and } r= a,\\
0,& \text{ otherwise}.
\end{cases}
$$
We have $\ds \lgo_{z}(n)=\prod_{p\mid n} \lgo_z(p^{v_p(n)})$ and $\eta_f(n)\eta_f(m)=\eta_f(nm)$ if $(n, m)=1,$  thus
\begin{equation*}
\lgo_{z}(n)=\prod_{p\mid n} \left(\sum_{x\in J_N(p^{v_p(n)})}c_{z, x}^N(p^{v_p(n)})\eta_{f}(x)\right)=\sum_{x\in J_N(n)}c_{z, x}^N(n)\eta_f(x),
\end{equation*}
with 
$$
c_{z, x}^N(n)=\prod_{p\mid n} c_{z, p^{v_p(x)}}^N(p^{v_p(n)}).
$$
Note that the coefficients $c_{z, x}^N(n),$ $I_N(n),$ and $J_N(n)$ depend only on the level $N$ and not directly on the modular form $f.$ Let us also define $I(n)=I_1(n),$ $J(n)=J_1(n),$ and $c_{z,x}(n)=c_{z, x}^1(n).$ They will employed in the statement of Theorem \ref{average_primitive}, which is our main result on averages over the set of primitive forms $B_k(N).$

Let $B(a, R)=\{z\in\C\mid |z-a|<R\}$ denote the open disc of radius $R$ and center $a\in \C,$ let $\overline{B(a,R)}$ be the corresponding closed disc. We also put $\calD_R=\overline{B(0, R)}.$ The following estimate is used throughout the paper.

\begin{lemma}
\label{MainEstimate}
For any $\epsilon>0$ and $z\in \calD_R$  we have $|c_{z, x}^N (n)|\ll_{\epsilon, R} n^\epsilon$ and $|\lgo_z(n)|\ll_{\epsilon,R} n^\epsilon.$
\end{lemma}

\begin{proof}
To see this, recall (\cite[3.1.2]{IM3}) that for any prime $p$
$$
\left|H_r\left(\frac{iz}{2}\right)\right|\leq H_{r}\left(\frac{|z|}{2}\right)\leq h_{r}\left(\frac{|z|}{2}\right)\leq h_{r}(|z|\log p)
$$ 
and 
$$
\left|h_r\left(-\frac{iz}{2}\log p\right)\right|\leq h_r(|z|\log p)\leq \exp\left(2\sqrt{r|z|\log p}\right),
$$
thus in both cases $|\hg_z(p^r)|\leq \exp\left(2\sqrt{r|z|\log p}\right)$. Using the concavity of the function $\sqrt{x}$, we see that
\begin{align*}
|c_{z, x}^N(p^r)|&\leq e^{2\sqrt{\frac{r-a}{2}|z|\log p}}e^{2\sqrt{\frac{r+a}{2}|z|\log p}}+e^{2\sqrt{(\frac{r-a}{2}-1)|z|\log p}}e^{2\sqrt{(\frac{r+a}{2}+1)|z|\log p}}\\
&\leq e^{2\sqrt{|z|\log p}\left(\sqrt{\frac{r-a}{2}}+\sqrt{\frac{r+a}{2}}\right)}+e^{2\sqrt{|z|\log p}\left(\sqrt{\frac{r-a}{2}-1}+\sqrt{\frac{r+a}{2}+1}\right)}\\
&\leq e^{2\sqrt{|z|\log p}\sqrt{2r}}+e^{2\sqrt{|z|\log p}\sqrt{2r}}\leq 2e^{2\sqrt{2r|z|\log p}}.
\end{align*}
when $p\nmid N.$ The above estimates on $\hg_z(p^r)$ also imply the same bound on $c_{z,x}^N(p^r)$ when $p\mid N.$

Now, denoting by $\omega(n)$ the number of distinct prime divisors of $n$ and using once again the concavity of $\sqrt{x},$ for $n=\prod\limits_{p\in \calP} p^{v_p(n)}$ we have
\begin{multline*}
\log |c_{z,x}^N(n)| \leq \sum_{p\mid n} (\log 2 + \sqrt{v_p(n)\log p}\sqrt{8R}) \ll_R \left(\sum_{p\mid n} \sqrt{v_p(n)\log p}\right)\sqrt{8R} \\
\ll_R \sqrt{\sum_{p\mid n} v_p(n)\log p} \sqrt{\omega(n)}\ll_R \sqrt{\frac{\log n}{2+\log \log n}}\sqrt{\log n},
\end{multline*}
since  by \cite[Sublemma 3.10.5]{Ih1} (which is classical in the case of $\bbN$) we have
\begin{equation}
\label{boundomega}
\omega(n)\ll \frac{\log n}{2+\log \log n}.
\end{equation}
We thus conclude that $|c_{z, x}^N (n)|\ll_{\epsilon, R} n^\epsilon.$

As for the second statement, we notice that the estimate \eqref{sizeJ} together with Deligne bound $|\eta_{f}(n)|\leq \tau(n) \ll_\epsilon n^\epsilon$ imply
$
\lgo_z(n)\ll_\epsilon |J_N(n)| \cdot n^\epsilon \cdot \tau(n) \ll_\epsilon n^{3\epsilon}.
$
\end{proof}

We conclude the section by the following trivial but useful lemma.

\begin{lemma}
\label{coefficients_conjugate}
We have $\overline{\lgo_z(n)}=\lgo_{-\bar{z}}(n),$ and $\overline{c_{z, x}^N(n)}=c_{-\bar{z}, x}^N(n).$
\end{lemma}

\begin{proof}
The eigenvalues $\eta_f(n)$ are all real, so the $L$-functions $L(f, s)$ have Dirichlet series with real coefficients. Thus the statement of the lemma follows from the definition of the coefficients $\lgo_z(n),$ and $c_{z, x}^N(n).$
\end{proof}

\section{Average on twists}
\label{section_twists}
 
This section is devoted to the proof of an averaging result for twists of a given primitive form. It is to a large extent based on the work of Ihara and Matsumoto \cite{IM3}, which provides a general setting for the problem we consider. 

\subsection{Setting} 

Let us fix a primitive cusp form $f\in B_k(N)$ of weight $k$ and level $N.$ Let $\chi\colon (\Z/m\Z)^\times\to \C^\times$ be a primitive character $\mathrm{mod}\ m,$ where $(m, N)=1.$ It is known (see \cite[Prop. 14.19 and Prop. 14.20]{IK}) that $f\otimes \chi$ is a primitive form of weight $k,$ level $Nm^2,$ and nebentypus $\chi^2.$ We consider the twisted $L$-function given by
$$
L(f\otimes \chi,s)=\prod_p L_{p}(f\otimes \chi, s),
$$ 
where the local factors are defined as follows:
$$
L_p(f\otimes  \chi,s)=\left(1-\alpha_f(p)\chi(p)p^{-s}\right)^{-1}\left(1-\beta_f(p)\chi(p)p^{-s}\right)^{-1},  
$$
with the notation of \S2. It is an $L$-function of degree $2$ and conductor $Nm^2,$ entire and polynomially bounded in vertical strips. After multiplication by the gamma factor 
$$
\gamma_{k}(s)=\sqrt{\pi}2^{\frac{3-k}{2}}(2\pi)^{-s}\Gamma\left(s+\frac{k-1}{2}\right),
$$ 
it satisfies a functional equation \cite[\S 5.11]{IK}. Its analytic conductor $\qg(f\otimes\chi,s)$ is defined as follows: 
$$
\qg(f\otimes\chi,s)=Nm^2\left(\left|s+\frac{k-1}{2}\right|+3\right)\left(\left|s+\frac{k+1}{2}\right|+3\right)\leq Nm^2(|s|+k+3)^2.
$$ 

% !!!Functional equation!!!
%, namely,
%put $$\gamma_{k}(s)=2(2\pi)^{-\frac{k-1}{2}}(2\pi)^{-s}\Gamma\left(s+\frac{k-1}{2}\right),$$ and consider $\Lambda(f\otimes \chi,s)=\gamma_{k}(s)L(f\otimes\chi,s).$ 
%$$q_f=m\sqrt{N}$$ then  
%et $$\epsilon_E(\chi)=\frac{\omega_E\chi(N_E)\tau(\chi)^2}{f}.$$
%This function satifies the functional equation
%$$\Lambda_E(s,\chi)=q_E(f)\gamma_E(s)L_E(s,\chi),$$ on a alors: 
%$$\Lambda(f\otimes\chi,s)=\Lambda(f\otimes\chi,1-s).$$

Just as in \S2 we use the following notation for modular forms with nebentypus:
$$
\begin{cases}
\ds g(f\otimes \chi,s,z)=\exp\left(\frac{iz}{2} \frac{L'(f\otimes \chi,s)}{L(f\otimes \chi,s)}\right),\\ 
\ds G(f\otimes \chi,s,z)=\exp\left(\frac{iz}{2}\log L(f\otimes \chi,s)\right).
\end{cases}
$$
We also write $\ggo(f\otimes \chi,s,z)$ to denote either of the above two functions.

If $G$ is a function on a finite group $K,$ let $\Avg_{\chi\in K}G(\chi)$ denote the usual average $|K|^{-1}\sum_{\chi\in K}G(\chi).$

\subsection{The \texorpdfstring{$\tilde{M}$}{M-tilde}-function}

We would like to understand the average over all Dirichlet characters $\mod m$ of the functions $\ggo(f\otimes \chi,s,z),$ when $m$ runs through large prime numbers. Ihara and Matsumoto's results apply in this case and we get the following theorem.

\begin{theo} 
\label{averagechar}
Assume that $m$ is a prime number. Let $\Gamma_{m}$ denote the group of Dirichlet characters modulo $m.$ Let $0<\epsilon<\frac{1}{2}$ and $T,R>0.$ Let  $s=\sigma+it$ belong to the domain $\sigma\geq\epsilon+\frac{1}{2},$ $|t|\leq T,$ let $z$ and $z'$ be inside the disk $\calD_R$. Then, assuming the Generalized Riemann Hypothesis (GRH) for $L(f\otimes\chi, s)$, in the notation of \S2 we have
\begin{equation}
\label{avchar}
\Avg_{\chi\in\Gamma_{m}}\left(\overline{\ggo(f\otimes\chi,s,z)}{\ggo(f\otimes\chi,s,z')}\right)-\sum_{(n,m)=1} \overline{\lgo_{z}(n)}\lgo_{z'}(n)n^{-2\sigma}\ll_{\epsilon,R,T, f} m^{-\frac{\epsilon}{2}}.
\end{equation}
Moreover,
$$
\lim_{m\to \infty}\Avg_{\chi\in\Gamma_{m}}\left(\overline{\ggo(f\otimes\chi,s,z)}{\ggo(f\otimes\chi,s,z')}\right)=\sum_{n\geq 1}\overline{\lgo_{z}(n)}\lgo_{z'}(n)n^{-2\sigma}.
$$
\end{theo}

\begin{proof} 

We notice that $\ggo(f\otimes\chi,s,z)=\sum\limits_{n\geq 1} \lgo_z(n)\chi(n) n^{-s}$, where $\lgo_z(n)$ are the coefficients of $\ggo(f, s, z).$ We thus can deduce the theorem from \cite[Theorem 1]{IM3}. We can pass to the situation treated in \cite{IM3} by omitting the summand corresponding to the trivial character $\chi_0$ since in our case all the $\ggo(f\otimes \chi, s, z)$ are holomorphic for $\Re s > \frac{1}{2}$. Thus, it is enough to prove that the family $\lgo_{|z|\leq R}$ is uniformly admissible in the sense of Ihara and Matsumoto.  

First of all, the property (A1), asserting that $\lgo_{|z|\leq R}(n)\ll_{\epsilon} n^\epsilon,$ follows from Lemma \ref{MainEstimate}.

The property (A2) states that $\ggo(f\otimes\chi,s,z)$ extend to holomorphic functions on $\Re s > \frac{1}{2}$ for any non trivial $\chi,$ which is true under GRH.

The property (A3) will be proven in the following lemma that will be used again in \S \ref{section_primitive}.

% One should be careful when $f$ is non-primitive! The lemma is to be derived separately by reducing the level.
\begin{lemma} 
\label{estimateL}
Let $f$ be a primitive form of weight $N$, and let $\chi$ be a primitive Dirichlet character of conductor $m$ coprime with $N.$ Then, assuming GRH for $L(f\otimes\chi,s),$ we have for $\Re s  \geq \frac{1}{2}+\epsilon:$
$$
\max(0,\log|\ggo(f\otimes\chi,s,z)|)\ll_{\epsilon,R}\ell(t)^{1-2\epsilon}\ell(mNk)^{1-2\epsilon},
$$
\end{lemma}
where $\ell(x)=\log(|x|+2),$ $t=\Im s.$
\begin{proof}[Proof of the Lemma]
First, the following estimates hold (\cite[Theorems 5.17 and 5.19]{IK}) for any $s$ with $\frac{1}{2}<\Re s=\sigma\leq \frac{5}{4}:$
$$
-\frac{L'(f\otimes\chi,s)}{L(f\otimes\chi,s)}= O\left(\frac{1}{2\sigma-1}(\log \qg(f\otimes \chi, s))^{2-2\sigma}+\log\log\qg(f\otimes \chi, s)\right),
$$ 
and 
$$
\log L(f\otimes \chi,s) = O\left(\frac{(\log\qg(f\otimes \chi, s))^{2-2\sigma}}{(2\sigma-1)\log\log\qg(f\otimes \chi, s)}+\log\log\qg(f\otimes \chi, s)\right),
$$ 
the implied constants being absolute. 

Next, for the same range of $s$ we have
$$
\log \qg(f\otimes\chi, s)\ll \log (mNk) + \log (|t|+2)\ll \ell(mNk) + \ell(t).
$$ 
Thus we see that 
\begin{multline*}
\log|\ggo(f\otimes\chi,s,z)|=\log\left|\exp\left(\frac{iz}{2}\Lg(f\otimes\chi,s)\right)\right|=\Re\left(\frac{iz}{2}\Lg(f\otimes\chi,s)\right)\ll_R |\Lg(f\otimes\chi,s)|,
\end{multline*}
so
$$
\max(0,\log|\ggo(f\otimes\chi,s,z)|)\ll_{\epsilon,R}\ell(t)^{1-2\epsilon}\ell(Nmk)^{1-2\epsilon}.
$$

If $\sigma\geq \frac{5}{4}$ a much simpler estimate suffices. Indeed, using the fact that \cite[(5.25)]{IK}
$$
-\frac{L'(f,s)}{L(f, s)}=\sum_n \frac{\Lambda_f(n)}{n^s} \quad  \text{ and } \quad \log L(f, s)=-\sum_n \frac{\Lambda_f(n)}{n^s\log n},
$$
with $\Lambda_f(n)$ supported on prime powers and $\Lambda_f(p^n)=(\alpha_f(p)^n + \beta_f(p)^n)\log p,$ we see that both $\ds\frac{L'(f,s)}{L(f, s)}$ and $\log L(f\otimes \chi,s)$ are bounded by an absolute constant. Thus the conclusion of the lemma still holds in this case.
\end{proof}

Thus Ihara and Matsumoto's property (A3) is established (with a stronger bound than required), since in our case $N$ and $k$ are fixed. So, the family we consider is indeed uniformly admissible.
\end{proof}

\begin{remark}
The estimate \eqref{avchar} should still be true if we omit the condition on $m$ to be prime. To prove it one establishes an analogue of Lemma \ref{estimateL}, replacing $\chi$ with the primitive character by which it is induced and estimating the bad factors of the $L$-function (with some additional work required when $m$ is not coprime with $N$). Then one uses once again \cite[Theorem 1]{IM3}, in which the first inequality is true without any restriction on the conductor.
\end{remark}

% This is not stated in the MMJ paper but is poven in the case without MF in the first paper by Ihara. Might be not so easy after all.
\begin{remark}
\label{average_uncond}
The theorem should hold unconditionally for $\sigma=\Re s> 1$ by orthogonality of characters, all the series being absolutely convergent in this domain.
\end{remark}

As a direct consequence, we obtain the following result on averages of the values of $\ggo$. Put 
$$
\tilde{M}_{s}(z_{1},z_{2})=\sum\limits_{n=1}^\infty\lgo_{z_{1}}(n)\lgo_{z_{2}}(n)n^{-2s}.
$$ 
Because of Lemma \ref{MainEstimate}, the series converges uniformly and absolutely on $\Re s\geq \frac{1}{2}+\varepsilon,$ $|z_{1}|,|z_{2}|\leq R,$ defining a holomorphic function of $s,z_{1},z_{2}$ for $\Re s> \frac{1}{2}.$ Put 
$$
\psi_{z_{1},z_{2}}(w)=\exp\left(\frac{i}{2}(z_{1}\overline{w}+z_{2}w)\right).
$$ 

\begin{coro} 
\label{corochar}
Let $m$ run over prime numbers. Then, assuming GRH,
$$
\lim_{m\to\infty}\Avg_{\chi\in\Gamma_{m}}\psi_{z_{1},z_{2}}(\Lg(f\otimes\chi,s))=\tilde{M}_{\sigma}(z_{1},z_{2}).
$$
\end{coro}
\begin{proof}  By definition, we have : 
\begin{align*}
\psi_{z_{1},z_{2}}(\Lg(f\otimes\chi,s))&=\exp\left(\frac{i}{2}z_{1}\overline{\Lg(f\otimes\chi,s)}\right)\exp\left(\frac{i}{2}z_{2}\Lg(f\otimes\chi,s)\right)\\ &=\overline{\ggo(f\otimes\chi,s,-\bar{z}_{1}}){\ggo(f\otimes\chi,s,z_{2})}.
\end{align*}
By Theorem \ref{averagechar} we get 
$$
\lim_{m\to \infty}\Avg_{\chi\in\Gamma_{m}}\psi_{z_{1},z_{2}}(\Lg(f\otimes\chi,s))=\sum_{n\geq 1}\overline{\lgo_{-\bar{z}_1}(n)}\lgo_{z_{2}}(n)n^{-2\sigma}.
$$
Lemma \ref{coefficients_conjugate} implies that $\overline{\lgo_{-\bar{z}}(n)}={\lgo_z(n)},$ so the corollary is proven.
\end{proof}

% !!! Some useless corollary 
%\begin{coro} Under the assumptions of the theorem we have
%$$
%\lim_{m\to\infty} \Avg_{\chi\in\Gamma_{m}}\left|L(f\otimes\chi,s)^\frac{iz}{2}\right|^2=\tilde{M}_{\sigma}(-\bar{z},z).
%$$
%\end{coro}

\section{The distribution of \texorpdfstring{$L$}{L}-values for twists}
\label{section_distribution}

Our next result concerns the distribution of the values of logarithmic derivatives and logarithms of $L$-functions of twists of a fixed modular form $f$. In this section the dependence on $f$ in $\ll$ will be omitted.

Recall that we have defined 
$$
\tilde{M}_{s}(z_{1},z_{2})=\sum\limits_{n=1}^\infty\lgo_{z_{1}}(n)\lgo_{z_{2}}(n)n^{-2s},
$$
the corresponding series being absolutely and uniformly convergent on $\Re s \geq    \frac{1}{2}+\epsilon, |z_1|\leq R, |z_2|\leq R.$ For $\sigma \in \bbR,$ we put $\tilde{M}_\sigma(z)=\tilde{M}_\sigma(z, \bar{z}).$

Define the family of additive characters
$$
\psi_{z_{1},z_{2}}(w)=\exp\left(\frac{i}{2}(z_{1}\overline{w}+z_{2}w)\right).
$$ 
We also let $\psi_z(w)=\psi_{z, \bar{z}}(w)=\exp(i\Re (z\bar{w})).$
Recall that the Fourier transform of $\phi\colon \C \to \C,$ $\phi \in L^1$ is defined as
\begin{equation*}
\calF \phi(z)=\int_{\C} \phi(w) \psi_{z}(w)|dw|= \frac{1}{2\pi}\int_{\C} \phi(w) e^{i\Re(z\bar{w})}|dw|=\frac{1}{2\pi}\int_{\bbR^2} \phi(w) e^{i(x x' + y y')}dx dy,  
\end{equation*}
where $\ds|dw|=\frac{1}{2\pi}dxdy, x=\Re w, y=\Im w, x'=\Re z, y'=\Im z.$

The goal is to prove the following equidistribution result, which is an analogue of \cite[Theorem 4]{IM3}.

%, as well as of the unconditional results from \cite[Theorem 1]{IM2} and \cite[Theorem 1.1]{IM4}. 

\begin{theo}
\label{maindistrib}
%\begin{enumerate}
%\item 
Let $\Re s = \sigma>\frac{1}{2}$ and let $m$ run over prime numbers. Let $\Phi$ be either a continuous function on $\C$ with at most exponential growth, that is $\Phi(w)\ll e^{a|w|}$ for some $a>0,$ or the characteristic function of a bounded subset of $\C$ or of a complement of a bounded subset of $\C.$ Define $M_\sigma$ as the inverse Fourier transform of $\tilde{M}_\sigma(z),$ $M_\sigma(z)=\calF \tilde{M}_\sigma(-z).$ Then under GRH for $L(f\otimes \chi,s)$ we have 
\begin{equation}
\label{Ldensity}
\lim_{m\to\infty}\Avg_{\chi\in\Gamma_{m}}\Phi(\Lg(f\otimes\chi,s))=\int_{\C}M_{\sigma}(w)\Phi(w)|dw|.
\end{equation}

%\item The equality \eqref{Ldensity} holds unconditionally either for the characteristic function of a compact subset of $\C$, of the complement of such a subset or for any bounded continuous function $\Phi$ on $\C.$ 
%\end{enumerate}
\end{theo}

\begin{remark} The above theorem should hold unconditionally for any $\sigma > 1$ and any continuous function $\Phi$ on $\C,$ by virtue of Remarks \ref{average_uncond} and (iv) of Corollary \ref{M_is_not_bad}.
\end{remark}

To prove this theorem we first construct the local $M$ and $\tilde{M}$-functions and establish their properties. We then obtain a convergence result for partial $M$-functions $M_{s,P}$ for finite sets of primes $P$ to a global function $M$. This allows us to prove some crucial estimates for the growth of $M.$ Finally, we deduce the global result using corollary \ref{corochar}. Our approach is strongly influenced by that of Ihara and Matsumoto, the main ingredients being inspired by the results of Jessen and Wintner \cite{JW} that we have to adapt to our situation. 

All the results below, except from the proof of Theorem \ref{maindistrib} itself, do not depend on GRH.

\subsection{The functions \texorpdfstring{$M_{s,P}$}{MsP} and \texorpdfstring{$\tilde{M}_{s, P}$}{MsP-tilde}.}

Let $\Re s = \sigma>0.$ Define the functions on $T_p=\C^1=\{t\in\C\ | \ |t|=1\}$ by  
\[
g_{s,p}(t)=\frac{-(\log p)\alpha(p) p^{-s}t}{1-\alpha(p)p^{-s}t}+\frac{-(\log p)\beta(p) p^{-s}t}{1-\beta(p)p^{-s}t},
\] 
and
$$
G_{s,p}(t)=-\log(1-\alpha(p)p^{-s}t)-\log(1-\beta(p)p^{-s}t).
$$
As before, we let $\ggo_{s,p}$ denote either $g_{s,p}$ or $G_{s,p},$ depending on the case we consider. We note that the local factor of the $L$-function is $1$ once $p^2\mid N,$ so we can omit such primes from our considerations. 

Denote by $\fg_{p}(z)$ the expression
$$
\frac{-(\log p)\alpha(p) z}{1-\alpha(p)z}+\frac{-(\log p)\beta(p)z}{1-\beta(p)z} \quad \text{or} \quad
-\log(1-\alpha(p)z)-\log(1-\beta(p)z).
$$
in the $\log'$ and $\log$ case respectively. Note that if $p\nmid N,$ $\fg_{p}(z)=\ds -\log p\cdot\frac{\eta_{f}(p)z-2z^2}{1-\eta_{f}(p)z+z^2}$ or $-\log(1-\eta_{f}(p)z+z^2)$ respectively. The functions $\fg_p(z)$ are holomorphic in the open disc $|z|<1.$ We obviously have $\ggo_{s,p}(t)=\fg_{p}(p^{-s}t).$

For a prime number $p$, let $T_{p}=\C^1$ be equipped with the normalized Haar measure $\ds d^\times t=\frac{dt}{2\pi i t} $. If $P$ is a finite set of primes, we let $T_{P}=\prod\limits_{p\in P}T_{p}$ and we denote by $d^\times t_{P}$ the normalized Haar measure on $T_{P}.$ Put also $\ds\ggo_{s,P}=\sum_{p\in P}\ggo_{s,p}.$

We introduce the local factors $\tilde{M}_{s, p}(z_1,z_2)$ via
\begin{equation}
\label{Msp}
\tilde{M}_{s,p}(z_{1},z_{2})=\sum_{r=0}^{+\infty}\lgo_{z_{1}}(p^r)\lgo_{z_{2}}(p^r)p^{-2rs}.
\end{equation}
The series is absolutely and uniformly convergent on compacts in $\Re s > 0$ by Lemma \ref{MainEstimate}. Put $\tilde{M}_{s, P}(z_1, z_2)=\prod\limits_{p\in P} \tilde{M}_{s, p}(z_1, z_2).$ We also define $\tilde{M}_{\sigma, p}(z)=\tilde{M}_{\sigma, p}(z, \bar{z}),$ and $\tilde{M}_{\sigma, P}(z)=\tilde{M}_{\sigma, P}(z, \bar{z}).$

\begin{lemma}
\label{Fourier}
\begin{enumerate}
\item The function $\tilde{M}_{s, P} (z_1, z_2)$ is entire in $z_1, z_2.$

\item We have 
\begin{equation*}
\tilde{M}_{s,p}(z_{1},z_{2})=\int_{\C^1}\exp\left(\frac{i}{2}(z_{1}\ggo_{s,p}(t^{-1})+z_{2}\ggo_{s,p}(t))\right)d^\times t.
\end{equation*}
% =\int_{\C^1}\psi_{z_1, z_2}(\ggo_{\sigma,p}(t))d^\times t (for real s)
In particular,
$$
\tilde{M}_{\sigma,p}(z_1,z_2)=\int_{\C^1}\psi_{z_1, z_2}(\ggo_{\sigma,p}(t))d^\times t, \quad \text{ and } \quad \tilde{M}_{\sigma,p}(z)=\int_{\C^1}\exp(i\Re(\ggo_{\sigma,p}(t)\bar{z}))d^\times t.
$$

\item The ``trivial'' bound $|\tilde{M}_{\sigma,p}(z)|\leq 1$ holds. 

\end{enumerate}
\end{lemma}

\begin{proof}
(i) This is a direct corollary of the absolute and uniform convergence of the series of analytic functions \eqref{Msp}, defining $\tilde{M}_{s, p}(z_1, z_2).$

(ii) It is clear from the definitions that $\ds\exp\left(\frac{iz}{2}\ggo_{s,p}(t)\right)=\sum_{r=0}^{\infty}\lgo_{z}(p^r)(p^{-s}t)^r.$ So, the statement is implied by the fact that $\tilde{M}_{s,p}$ is the constant term of the Fourier series expansion of $\exp\left(\frac{i}{2}(z_{1}\ggo_{s,p}(t^{-1})+z_{2}\ggo_{s,p}(t))\right).$

(iii) Obviously follows from (ii).
\end{proof}

For the sake of convenience in what follows we will identify a function on $\bbR^2$ with the Radon measure or the tempered distribution it defines, when the latter make sense.  We will also regard the Fourier transform or the convolution products as being defined via the corresponding distributions. We refer to \cite[\S 2, \S 3]{JW} for more details.

\begin{prop}  
\label{localmain}
\begin{enumerate}
\item There exists a unique 
%distribution 
positive measure $M_{\sigma,P}$ of compact support and mass $1$ on $\C\simeq \mathbb{R}^2$  such that 
$$
%\int_{\C}M_{\sigma,P}(w)\Phi(w)|dw|
M_{\sigma,P}(\Phi)=\int_{T_{P}}\Phi(\ggo_{s,P}(t_{P}))d^\times t_{P}
$$ 
for any continuous function $\Phi$ on $\C.$  

\item  $\calF M_{\sigma,P} = \tilde{M}_{\sigma,P}(z).$ 

\item There exists a set of primes $\calP_f$ of positive density such that, for all $p\in \calP_{f},$ $\tilde{M}_{\sigma, p}(z) \ll_{p,\sigma} (1+|z|)^{-\frac{1}{2}}.$ 

\item Let $P$ be a set of primes. If $|P\cap \calP_{f}|>4$, then $M_{\sigma,P}$ admits a continuous density (still denoted by $M_{\sigma,P}$) which is an $L^1$ function. The function $M_{\sigma,P}$ satisfies $M_{\sigma,P}(z)=M_{\sigma,P}(\bar{z})\geq 0.$ 

\item $M_{\sigma,P}$ is of class $\calC^r$ once $|P\cap \calP_{f}| > 2(r+2).$
\end{enumerate}
\end{prop}

\begin{proof}
(i) The uniqueness statement is obvious and the existence is given by the direct image measure $(\ggo_{s,P})_*(d^\times t_{P}).$ The volume of an open set $U$ of $\mathbb{R}^2$ is thus given by $M_{\sigma,P}(U)=\Vol(\ggo_{s,P}^{-1}(U)),$ therefore $M_{\sigma,P}$ has compact support equal to the image of $\ggo_{s, P}$ and mass $1.$ From the formula $\ds M_{s,P}(\Phi)=\int_{T_{P}}\Phi(\ggo_{s,P}(t_{P}))d^\times t_{P},$ it is clear that $M_{s,P}$ depends only on $\sigma,$ since Haar measures on $T_P$ are invariant under multiplication by $p^{i\Im(s)}.$ 

(ii) From the definition of the convolution product we note that, regarded as distributions with compact support, $M_{\sigma,P}=\ast_{p\in P}M_{\sigma,p}.$

Next, $\mathcal{F}M_{\sigma,P}=\mathcal{F}(\ast_{p\in P}M_{\sigma,P})=\prod_{p\in P} \mathcal{F}M_{\sigma,p}.$ From Lemma \ref{Fourier} we see that $\tilde{M}_{\sigma,P}(z_{1},z_{2})=M_{\sigma,P}(\psi_{z_{1},z_{2}}),$ and for the Fourier transforms of tempered distributions on $\C\simeq \mathbb{R}^2$ we have 
\begin{align*}
\mathcal{F}M_{\sigma,p}(\phi)&=M_{\sigma,p}\left(\int_{\C}\psi_{z}(w)\phi(w)|dw|\right)=\int_{T_{p}}\int_{\C}\psi_{g_{s,p}(t)}(w)\phi(w)|dw|d^\times t\\
&=\int_{\C}\int_{T_{p}}\psi_{g_{s,p}(t)}(w)\phi(w)d^\times t|dw|=\int_{\C}M_{\sigma,p}(\psi_{z}(w))\phi(w)|dw|\\
&=\int_{\C}M_{\sigma,p}(\psi_{w}(z))\phi(w)|dw|=\int_{\C}\tilde{M}_{\sigma,p}(w)\phi(w)|dw|.
\end{align*}
We deduce that $\calF M_{\sigma, P}=\tilde{M}_{\sigma,P}(z).$ 

(iii) This is the most delicate part. Unfortunately, we cannot apply Jessen--Wintner theorem \cite[Theorem 13]{JW} to $\fg_{p}(z),$ since $\rho_{0}$ (in the notation of the latter theorem) depends on $p.$ Therefore, we need to establish the following explicit version of their result.

\begin{lemma} Let $\rho>0$ and let $\ds F(z)=\sum_{k\geq 1} a_{k} z^k$ be absolutely convergent for $|z|< \rho+\epsilon,$ $\epsilon> 0.$ Let $S\subset \C$ denote the parametric curve $\{S(\theta)\}_{\theta\in[0,1]}=\{F(re^{2\pi i \theta})\}_{\theta\in[0,1]}.$ Let $D_r$ be the distribution on $\C=\mathbb{R}^2$ defined as the direct image of the normalized Haar measure on the circle of radius $r$ in $\C$ by $F$ and let $\tilde{D}_r=\calF D_r$ be its Fourier transform. Assume that $|a_{1}|\neq 0.$ Then, if
$$
\rho'''= \frac{|a_{1}|}{\sqrt{2}\left(\ds\sum_{k\geq 2}k^3|a_{k}|\rho^{k-2}\right)},
$$ 
for any $r<\rho_{0}=\min(\rho, \rho''')$ we have $\tilde{D}_r(z)\ll_{r, F} (1+|z|)^{-\frac{1}{2}}.$
\end{lemma}

\begin{proof} Our goal is to make the proof of \cite[Theorem 13]{JW} explicit in order to be able to estimate $\rho_{0}.$ To do so, we will verify the conditions of \cite[Theorem 12]{JW} by proceeding in several steps.

First of all, we want to ensure that $F'(z)\neq 0,$ and the curve $S$ is Jordan. Put 
$$
\ds\rho'=\frac{|a_{1}|}{\ds \sqrt{2}\sum_{k\geq 2}k|a_{k}|\rho^{k-2}}.
$$ 
If $r<\min(\rho,\rho'),$ we have $F'(z)\neq 0$ for all $z\in \calD_r=\overline{B(0,r)},$ and $F$ is injective on $\calD_r.$ Indeed, either $|\Re a_{1}|$ or $|\Im a_{1}|$ is greater than $\ds\frac{|a_{1}|}{\sqrt{2}}.$ Without loss of generality we can suppose that $\ds|\Re a_{1}|\geq \frac{|a_1|}{\sqrt{2}}.$ Then 
$$
|\Re F'(z)|\geq |\Re a_{1}|-|z|\sum_{k\geq 2}k|a_{k}|\rho^{k-2}\geq \frac{|a_{1}|}{\sqrt{2}}-|z|\sum_{k\geq 2}k|a_{k}|\rho^{k-2}>0
$$ 
on $\calD_r,$ in particular $F'(z)\neq 0.$ The sign of $\Re F'(z)$ does not change as the function is continuous, so once more, without loss of generality, we may assume that $\Re F'(z)>0.$ Then, for $z_{1}\neq z_{2}$ two points in $\calD_r,$ we have by convexity of $\calD_r$, 
$$
\ds\Re\frac{F(z_{2})-F(z_{1})}{z_{2}-z_{1}}=\ds\int_{0}^1\Re F'(z_{1}+t(z_{2}-z_{1}))dt>0,
$$ 
which proves the injectivity. Thus $F$ is a conformal transformation and $S$ is a Jordan curve.

The next step is to get a condition for the curve $S$ to be convex.
We use a well-known criterion \cite[Part 3, Chapter 3, 108]{PS}, stating that $S$ is convex if $$
\ds \Re \frac{z F''(z)}{F'(z)}> -1
$$ 
on $|z|=r.$ The estimate 
$$
\left| \Re \frac{z F''(z)}{F'(z)} \right| \leq \frac{|zF''(z)|}{|F'(z)|}\leq \frac{|z|\sum_{k\geq 2} k(k-1)|a_k| \rho^{k-2}}{|a_1|-|z|\sum_{k\geq 2} k|a_k| \rho^{k-2}}\leq  \frac{|z|\sum_{k\geq 2} k(k-1)|a_k| \rho^{k-2}}{|a_1| \left(1-\frac{1}{\sqrt{2}}\right)}
$$
for $r<\min(\rho, \rho')$  implies that the condition is satisfied once the left-hand side is less than one, that is
$$
r< \rho'' = \frac{|a_1|(2-\sqrt{2})}{2\sum_{k\geq 2} k(k-1)|a_k| \rho^{k-2}}.
$$

%%%%%%%% Another proof that S is convex %%%%%%%%%%%
%%%%%%%%%%%%%%%%%%%%%%
\begin{comment}
 To show that $S$ is convex, it suffices to show that the image of $\calD_r$ under $F$ is convex. 
This can be seen directly from the Newton method, but we will use \cite[Theorem 2.1]{Pol}. Indeed, for any $x,y\in B(0,r),$  
$$
|F'(x)-F'(y)|\leq L |x-y|\quad \text{ with } L=|a_{1}|+\sum_{k=2}^{+\infty}j(j-1)|a_{j}|\rho^{j-2}
$$ 
and $|\overline{F'(0)}w|=|a_{1}||w|,$ implying that the image of the $B$ is convex provided $r<|a_{1}|/2L.$ 
% Together with the other inequalities on r
\end{comment}
%%%%%%%%%%%%%%%%%%%%%%

Now, the condition (i) of \cite[Theorem 12]{JW} is satisfied for all $r<\rho.$ As for (ii) we consider the function 
$$
g_{\tau}(\theta)=\sum_{k\geq 1} |a_{k}|r^k\cos 2\pi(k\theta+\gamma_{k}-\tau),
$$ 
where $\tau\in [0,1)$ is fixed and $a_{k}=|a_{k}|e^{2\pi i \gamma_k}.$ We have to prove that for $r$ explicitly small enough, its second derivative has exactly two roots on $[0,1).$ We compute
$$
h_\tau(\theta)=-\frac{g_\tau''(\theta)}{4\pi^2r}=|a_{1}|\cos 2\pi(\theta+\gamma_{1}-\tau)+r\sum_{k\geq 2}j^2|a_{k}|r^{k-2}\cos 2\pi(k\theta+\gamma_{k}-\tau),
$$ 
so
$$
h_\tau'(\theta)=-2\pi|a_{1}|\sin2\pi(\theta+\gamma_{1}-\tau)-2\pi r\sum_{k\geq 2}k^3|a_{k}|r^{k-2}\sin 2\pi(k\theta+\gamma_{k}-\tau).
$$ 
Take now 
$$
r<\frac{|a_{1}|}{\sqrt{2}\left(\ds\sum_{k\geq 2}k^3|a_{k}|\rho^{k-2}\right)}=\rho'''.
$$ 
Since $\ds\sum_{k\geq 2}k^3|a_k|\rho^{k-2}\geq \sum_{k\geq 2}k^2|a_{k}|\rho^{k-2},$ the function $h_\tau$ can possibly have zeroes only on the two intervals (modulo $1$) containing $\pm\frac{1}{4}-\gamma_{1}+\tau \mod 1$ defined by the condition $|\cos 2\pi(\theta+\gamma_{1}-\tau)|<\frac{1}{\sqrt{2}}.$ The same argument shows that $h_\tau$ is positive at $\theta=-\gamma_{1}+\tau\mod 1$ and negative at $\theta=\frac{1}{2}+\tau-\gamma_{1}\mod 1,$ and therefore it has at least one zero in each of these intervals.

On the other hand, when $|\cos 2\pi(\theta+\gamma_{1}-\tau)|<\frac{1}{\sqrt{2}},$ we see that
$$
|h_\tau'(\theta)|\geq 2\pi|a_{1}||\sin2\pi(\theta+\gamma_{1}-\tau)|-2\pi r\sum_{k\geq 2}k^3|a_{k}|r^{k-2}>2\pi|a_{1}|\left(\sqrt{1-\frac{1}{2}}-\frac{1}{\sqrt{2}}\right)=0,
$$ 
showing that there is exactly one zero of $h_\tau$ in each of the above intervals. 

We thus can apply \cite[Theorem 12]{JW}, obtaining that the conclusion of the theorem holds for $r< \rho_{0}=\min(\rho,\rho',\rho'',\rho''')= \min(\rho, \rho''').$
\end{proof}

%Caratheodory
 %Let $w_{1}=F(z_{1})$ and $w_{2}=F(z_{2}),$ with $z_{1},z_{2}\in B(0,r)$ and $\lambda\in [0,1].$ Put $w=\lambda w_{1}+(1-\lambda)w_{2}$ and $z_{0}=\lambda z_{1}+(1-\lambda)z_{2}.$ We have $w_{i}=F(z_{0})+F'(z_{0})(z_{i}-z_{0})+\varepsilon_{i}
 
%By \cite[Corollary 2 of Theorem 4]{Mur}, given $\delta>0$ there exists a set $P$ of positive density such that, for all $\ds p\in P,$ $|\eta_{f}(p)|> \sqrt{2}-\delta.$

By \cite[Corollary 2 of Theorem 4]{Mur}, there exists a set $P$ of positive density such that, for all $\ds p\in P,$ $|\eta_{f}(p)|> 1.$ We apply the above lemma to the functions $F=\mathfrak{f}_{p},$ $p\in P,$ defined by absolutely convergent series for $|z|<\rho+\epsilon$, with $\rho= \epsilon=\frac{1}{2},$ and to the radii $r_p=p^{-\sigma}.$ In the $\log$ case, the coefficient $|a_{1}|$ of the lemma is $|\eta_{f}(p)|,$ whereas we have for any $i,$ $|a_{i}|\leq 2.$ In the $\log'$ case, the coefficients are all multiplied by $\log p:$ $|a_{1}|$ is $|\eta_{f}(p)|\log p$ and $|a_{i}|\leq 2\log p.$ Thus, for $p$ such that $p\in P$ and 
$$
p^{-\sigma}<\frac{1}{\ds 8\sqrt{2}\sum_{k\geq 2}k^3 2^{-k}} = \frac{1}{204\sqrt{2}},
$$
% $\sum_{k=0}^{\infty} k^3 2^{-k} = 26$ (Wolfram Alpha), thus $p^{\sigma}\geq 289$ or $p geq 289^2$ suffices. 
we have that $\tilde{M}_{\sigma,p}(z)=O\left((1+|z|)^{-\frac{1}{2}}\right),$ proving thus (iii).

%Applying Jessen--Wintner theorem \cite[Theorem 13]{JW} to  $\fg_{p}(z),$ for  sufficiently large $p$  ($p^{-\sigma}<\rho_{0}$ in the notation of the latter theorem), satisfying $\eta_{f}(p)\neq 0,$ we get that $\tilde{M}_{s,p}(z,\bar{z})=O((1+|z|)^{-\frac{1}{2}}).$

%(iv) and (v) By the Fourier inversion formula, we get $\mathcal{F}\tilde{M}_{\sigma,P}(z,\bar{z})\circ (-\id)=M_{\sigma,P}.$ Thus, to deduce the regularity properties of $M_{\sigma,P}$ it suffices to bound the growth of $\tilde{M}_{\sigma,P}(z,\bar{z}).$

%For the primes $p$ satisfying either $\eta_{f}(p)= 0,$ $p\mid N,$ or being less than $P_0(f),$ we use the trivial bound $|\tilde{M}_{\sigma,p}(z,\bar{z})|\leq 1$ from Lemma \ref{Fourier}. For all the other primes $p$ the bound from (iii) can be applied. 
%\end{proof}

%!!!!!!Applying Jessen--Wintner theorem \cite[Theorem 13]{JW} to  $\fg_{p}(z),$ for  sufficiently large $p$  ($p^{-\sigma}<\rho_{0}$ in the notation of the latter theorem), satisfying $\eta_{f}(p)\neq 0,$ we get that $\tilde{M}_{s,p}(z)=O((1+|z|)^{-\frac{1}{2}}).$

(iv), (v) By the Fourier inversion formula, we get $\mathcal{F}\tilde{M}_{\sigma,P}(-z)=M_{\sigma,P}.$ It is well-known \cite[\S 3]{JW} that $f=\calF g$ is absolutely continuous and admits continuous density, once the integral $\int_{\C} |g(w)| |dw|$ converges. Moreover, it possesses continuous partial derivatives of order $\leq p,$ if the convergence holds for $\int_{\C} |z|^{p} |g(w)| |dw|.$ Thus, to deduce the regularity properties of $M_{\sigma,P}$ it suffices to bound the growth of $\tilde{M}_{\sigma,P}(z).$ 

For the primes $p\notin \calP_{f},$ we use the trivial bound $|\tilde{M}_{\sigma,p}(z)|\leq 1$ from Lemma \ref{Fourier}. For all the other $p$ the bound from (iii) can be applied. 

Now, the identity $M_{\sigma,P}(z)=M_{\sigma,P}(\bar{z})$ is the consequence of (ii) together with the symmetry $\tilde{M}_\sigma(z,\bar{z})= \tilde{M}_\sigma(\bar{z}, z).$ The positivity of $M_{\sigma,P}(z)$ follows from the definition $M_{\sigma,P}(U)=\Vol(\ggo_{\sigma,P}^{-1}(U))$ together with the continuity that we have established.
% !!!Another plausible argument for $M(z)=M(\bar(z))$
% $\overline{\ggo_{\sigma,P}(t)}=\ggo_{\sigma, P}(\bar{t}),$ and the invariance of the Haar measure $d^\times t_P$ under conjugations. 
\end{proof}

Let $\calP$ denote the set of all prime numbers, $\calP_x=\{p\in \calP \mid p\leq x\}.$

\begin{coro} 
\label{nonzero}
Given $r>0,$ $y>0,$ there exists $C=C(y, r, f)$ such that $\tilde{M}_{\sigma, \calP_x\setminus \calP_y}(z)=O((1+|z|)^{-r})$ and the function $M_{\sigma, \calP_x\setminus \calP_y}(z)$ is of class $\calC^r$ for all $x\geq C.$
%once $|\calP_x\setminus \calP_y|>C.$
\end{coro}

\begin{proof} 
This comes directly from the fact that $\calP_{f}$ has positive density, implying that there exists $C,$ such that if $x\geq C,$ then $(\calP_{x}\setminus\calP_{y})\cap \calP_{f}$ contains more than $2r+4$ primes.
% !!! Not strong enough for what we need
%If $f$ is non CM, the number of $p\leq x$ with $\eta_f(p)=0$ is $O(x/(\log x)^{3/2-\delta})$ for any $\delta >0$ by \cite[Corollaire 2, p. 374]{Se0}. 
%If $f$ is CM, one notes that $\eta_p(f)=0$ if and only if $p$ is inert in the quadratic extension $L/\Q$ corresponding to $f$ by the results of \cite{Rib}. Thus the number of primes with $\eta_f(p)\neq 0$ is $x/(2\log x) + O(x/(\log x)^2).$
\end{proof}

\begin{remark}
The previous proposition is motivated by the following equidistribution result that is essentially implied by \cite[Lemma 4.3.1]{Ih1} applied to $\Psi=\Phi \circ \ggo_{\sigma, P}:$
$$
\lim_{m\to \infty} \Avg_{\chi\in \Gamma_m} \Phi(\Lg_P(f\otimes\chi, s))=\int_{T_P} \Phi(\ggo_{\sigma, P}(t_P))d^\times t_P,
$$
where $\Phi$ is a an arbitrary continuous function on $\C,$ $\Lg_P$ is either the logarithm or the logarithmic derivative of the corresponding partial product $\prod_{p\in P} L_p(f\otimes \chi, s)$ for $L(f\otimes \chi, s),$ and $\chi$ runs through all Dirichlet character of prime conductor $m \not\in P.$ Note a difference in the type of average considered in the aforementioned lemma with the one we use. The proof stays the same, being an application of Weyl's equidistribution criterion together with the orthogonality of characters.

Note, however, that it is not at all obvious to pass from the local equidistribution result to the global one. This also seems to give (after very significant effort) only a certain weaker form of global averaging results (e.g.\cite{Ih1}, \cite{IM1}). Following later papers by Ihara and Matsumoto, we use instead the convergence for particular test functions (quasi-characters, c.f. Theorem \ref{averagechar}) and then deduce the general case, using the information on the resulting distributions together with some general statements on convergence of measures.
\end{remark}

\subsection{Global results for \texorpdfstring{$\tilde{M}_{\sigma}$}{Msigma-tilde}.}

Let us establish some global properties of $\tilde{M},$ in particular the convergence of $\tilde{M}_{\sigma,P}$  to $\tilde{M}_{\sigma}.$ From now on we assume that $\Re s =\sigma > \frac{1}{2},$ without mentioning it in each statement.

\begin{prop}
\label{Euler_product}
\begin{enumerate}
\item The function $\tilde{M}_s(z_1, z_2)$ is entire in $z_1, z_2.$

\item We have the Euler product expansion 
$$
\tilde{M}_{s}(z_{1},z_{2})=\prod_{p}\tilde{M}_{s,p}(z_{1},z_{2}),
$$ 
which converges absolutely and uniformly on $\Re s\geq\frac{1}{2}+\epsilon$ and $|z_{1}|,$ $|z_{2}|\leq R,$ for any $\epsilon,R>0.$ 
\item  $\tilde{M}_{\sigma}(z)=O((1+|z|)^{-N})$ for all $N>0.$
\end{enumerate}
\end{prop}

\begin{proof}
(i) This is a direct corollary of the absolute and uniform convergence of the series of analytic functions, defining $\tilde{M}_{s}(z_1, z_2).$

(ii) To prove the uniform convergence of the infinite product it is enough to establish it for the sum $\ds\sum_{p} |\tilde{M}_{s, p}(z_1, z_2)-1|.$
By Lemma \ref{MainEstimate} we see that
\begin{multline*}
|\tilde{M}_{s,p}(z_{1},z_{2})-1|\leq\sum_{r=1}^{\infty}|\lgo_{z_{1}}(p^r)| |\lgo_{z_{2}}(p^r)| p^{-2r\sigma}
\ll_{\epsilon', R} \sum_{r=1}^{\infty} p^{(2\epsilon'-2\sigma)r}\leq \sum_{r=1}^{\infty} p^{(-1-\epsilon)r}<2p^{-1-\epsilon}. 
\end{multline*}
which implies the convergence.

The limit of the infinite product equals $\tilde{M}_{s}.$ Indeed, the series for $\tilde{M}_s$ converges absolutely and uniformly, thus the difference between $\tilde{M}_s$ and the partial product over primes $p\leq x,$ which is $\ds\sum_{n\in S_{x}}\lgo_{z_1}(n)\lgo_{z_2}(n)n^{-2s}$, where $S_{x}$ is the set of integers $n$ divisible by at least one prime strictly greater than $x,$ tends to $0$ as $x\to\infty.$

(ii) Note that for any two sets $P\subset P'$ of primes, any $z\in \C,$ 
$$
|\tilde{M}_{\sigma,P'}(z)|\leq |\tilde{M}_{\sigma,P}(z)|.
$$ 
Corollary \ref{nonzero} implies that one can find a finite set of primes $P$ such that $\tilde{M}_{\sigma,P}(z)\ll {(1+|z|)}^{-N}.$ This is enough to conclude.
\end{proof}

\begin{remark}
Along the same lines as in \cite[3.20]{IM2}, one proves a more precise estimate:
$\ds|\tilde{M}_{\sigma,p}(z)-1|\ll |z|^2 p^{-2\sigma}$ in the $\log$ case, and $\ds|\tilde{M}_{\sigma,p}(z)-1|\ll |z|^2 p^{-2\sigma}\log p$ in the $\log'$ case with absolute constants in $\ll.$
\end{remark}
 
\begin{remark}
One should be able to write an explicit power series expansion of $\tilde{M}_s(z_1,z_2)$ similar to the one in \cite[\S 4, Theorem $\tilde{M}$]{IM3}.
\end{remark}

\subsection{Global results for \texorpdfstring{$M_{\sigma}$}{Msigma}.}

\begin{prop} 
\label{M_Derivability}
The sequence $(M_{\sigma,\calP_x}(z))_{x\gg 0}$ converges uniformly (as continuous functions) to $M_{\sigma}(z) := \mathcal{F}\tilde{M}_{\sigma}(-z).$ Moreover, for a fixed $y,$ the sequence of continuous functions $(M_{\sigma,\calP_x\setminus \calP_y})_{x\gg 0}$ converges uniformly to the continuous function  $M_{\sigma}^{(y)}=\ast_{\calP \setminus \calP_y}M_{\sigma,p}:= \calF \left(\prod_{p\in\calP\setminus \calP_y}\tilde{M}_{\sigma,p}(-z)\right),$ and we have $M_{\sigma}(z)=M_{\sigma,\calP_y}\ast M_{\sigma}^{(y)}.$
\end{prop}

\begin{proof}
First of all, the notation $x\gg 0$ is used to make sure that all the elements of the sequence are continuous functions.

Fix $\epsilon>0.$ One can find a closed disk $\calD_r$ and  $x'$ large enough, so that for all $P'\supset \calP_{x'},$ 
$$
\ds\int_{\C\setminus \calD_r}|\tilde{M}_{\sigma,P'}(w)||dw|<\epsilon.
$$
The sequence $(\tilde{M}_{\sigma,\calP_x}(z)_{x\gg 0}$ converges uniformly to $\tilde{M}_{\sigma}(z)=\tilde{M}_{\sigma}(z)$ on $\calD_r$ by Proposition \ref{Euler_product}, thus we can find $x''$ large enough to guarantee for $x>\max(x', x''),$ 
$$
\|\calF\tilde{M}_{\sigma}(z)-\calF\tilde{M}_{\sigma,\calP_x}(z)\|_{\infty}<2\epsilon.
$$ 
This proves that $(\calF\tilde{M}_{\sigma,\calP_x}(z))_{x\gg 0}=(M_{\sigma,\calP_x}(-z))_{x\gg 0}$ converges uniformly to $\calF\tilde{M}_{\sigma}(z)=M_{\sigma}(-z).$ 

The same arguments apply if we remove $\calP_y$ from the set of all primes. Moreover, taking the Fourier transform of $\tilde{M}_{\sigma}=\tilde{M}_{\sigma,\calP_y}\times \prod_{p\notin \calP_y}\tilde{M}_{\sigma,p},$ we see that 
$M_{\sigma}(z)=M_{\sigma,\calP_y}\ast (\ast_{\calP\setminus \calP_y}M_{\sigma,p}).$
\end{proof}

\begin{coro}
\label{M_is_not_bad}
We have 
\begin{enumerate} 

\item $M_\sigma(z)=M_\sigma(\bar{z})\geq 0;$ 

\item $\ds\int_{\C}M_{\sigma}(z)|dz|=1;$

\item $M_\sigma(z)\in \calC^\infty$ and the partial derivatives of $M_{\sigma, \calP_x}$  converge uniformly to those of $M_\sigma;$

\item If $\sigma>1,$ the support of $M_\sigma$ is compact. 
\end{enumerate}
\end{coro}

\begin{proof}
(i) This is obvious from the corresponding properties of $M_{\sigma, P}.$

(ii) Using the identity $M_\sigma(z)=\calF \tilde{M}_\sigma(-z),$ we see that 
$
\ds\int_{\C}M_{\sigma}(z)|dz|=\tilde{M}_\sigma(0)=1.
$

(iii) We note that, given $p,$ there exists $y_0$ such that for $p>y_0,$ $M_{\sigma, \calP_y}$ has continuous partial derivatives up to order $p.$ Now, letting $\ds D^{(a,b)}=\frac{\partial^{a+b}}{\partial^a z\,\partial^b \bar{z}},$ we have $D^{(a,b)}(f\ast g)=(D^{(a,b)} f)\ast g,$ if $f$ admits the corresponding partial derivative. The statement now follows from the uniform convergence of $M_{\sigma, \calP_x\setminus\calP_y}^{(y)}$ to $M_\sigma^{(y)}.$

(iv) Indeed, by the uniform convergence of $M_{\sigma, P}$ to $M_\sigma,$ and the fact that the support of $M_{\sigma, P}$ is equal to the image of $\ggo_{s, P},$ it is enough to prove that the latter is bounded for $\sigma >1$. This is true since the series $\sum\limits_p p^{-\sigma}$ converges for $\sigma >1.$
\end{proof}

We will now obtain the rapid decay of $M_{\sigma}$ à la Jessen--Wintner by proving the following proposition, which is crucial for the proof of the main theorem of this section.

\begin{prop} 
\label{M_Growth}
For any $\lambda>0,$   $M_{\sigma}(z)=O_{\sigma, \lambda}(e^{-\lambda |z|^2}),$ as $|z|\to\infty.$ The same is true for all its partial derivatives.
\end{prop}
 
\begin{proof} We adapt the proof of Jessen--Wintner \cite[Theorem 16]{JW} to our specific case. The proof is based on an argument of Paley and Zygmund. 
 
Let $\sigma>\frac{1}{2}$ and $\lambda>0$ be fixed. Let $p_{1}<\dots<p_{i}\dots$ denote the sequence of all prime numbers. Write $P_{j}=\{p_{1},\dots,p_{j}\}.$ We have
$$
\fg_{p}(z)=\sum_{i\geq 1} a_{i,p}z^i,
$$ 
on the disk $B(0,1).$  By writing $1-\eta_{f}(p)z+z^2=(1-\alpha_f(p)z)(1-\beta_f(p)z),$ where $|\alpha_f(p)|$ and $|\beta_f(p)|$ are less than or equal to $1,$ we see that for all $i,$ $|a_{i,p}|\leq 2\log p$ in the $\log'$ case and $\leq 2$ in the $\log$ case respectively.
 
Put $r_{p}=p^{-\sigma}.$ Then the series $\ds\sum_{p} |a_{1,p}|^2r_{p}^2$ converges, so that we can find $q$ such that 
$$
d=1-2\lambda \sum_{p>p_{q}} |a_{1,p}|^2r_{p}^2>0.
$$

For $n>q$ let us look at the partial sums
$$
s_{n}(\theta_{1},\dots,\theta_{n})=\sum_{j=1}^n \fg_{p_{j}}(r_{p_{j}}e^{i\theta_{j}})\quad\text{and}\quad t_{n}(\theta_{{q+1}},\dots,\theta_{n})=\sum_{j=q+1}^na_{1,p_{j}}r_{p_{j}}e^{i\theta_{j}},
$$
where $\theta_j\in [0, 2\pi].$ We can bound the difference by
\begin{align*}
|s_{n}(\theta_{1},\dots,\theta_{n})-t_{n}(\theta_{q+1},\dots,\theta_{n})|&\leq \left|\sum_{j=1}^q \fg_{p_{j}}(r_{p_{j}}e^{i\theta_{j}})\right|+\sum_{j=q+1}^n\sum_{k=2}^{\infty} |a_{k,p_{j}}|r^k_{p_{j}}\\
&\leq \left |\sum_{j=1}^q \fg_{p_{j}}(r_{p_{j}}e^{i\theta_{j}})\right|+2\sum_{j=q+1}^n r^2_{{p_{j}}}(1-r_{p_{j}})^{-1}\log p_{j}\\
&\leq \left|\sum_{j=1}^q \fg_{p_{j}}(r_{p_{j}}e^{i\theta_{j}})\right|+8\sum_{j=q+1}^{+\infty}r^2_{{p_{j}}}\log p_{j}\\
&\leq \sum_{j=1}^q \sup\limits_{\vartheta_j\in[0, 2\pi]} |\fg_{p_{j}}(r_{p_{j}}e^{i\vartheta_{j}})|+8\sum_{j=q+1}^{+\infty}r^2_{{p_{j}}}\log p_{j}\ll A(q)
\end{align*}
as $\ds (1-r_{p_{j}})^{-1}\leq \frac{\sqrt{2}}{\sqrt{2}-1}.$ Here $A$ depends only on $q$ and not on $n$.

By an inequality of Jessen \cite[p. 290--291]{Je}, writing $|s_{n}|^2\leq 2|s_{n}-t_{n}|^2+2|t_{n}|^2,$ we obtain
\begin{multline}
\label{Jessenineq}
\int_{T_{P_{n}}}\exp(\lambda |s_{n}(\theta_{1},\dots,\theta_{n})|^2)d\theta_{1}\dots d\theta_{n}  \\
\leq e^{2\lambda A(q)^2} \int_{T_{P_{n,q}}}\exp(2\lambda |t_{n}(\theta_{q+1},\dots,\theta_{n})|^2)d\theta_{q+1}\dots d\theta_{n},\\
\leq \frac{e^{2\lambda A(q)^2}}{\ds 1-2\lambda\sum_{j={q+1}}^n |a_{1,p}|^2r_{p_{j}}^2}\leq e^{2\lambda A(q)^2}d^{-1}=K.
\end{multline}
where $P_{n,q}=P_{n}\setminus P_{q}.$ Noting that $M_{\sigma, P_n}(e^{\lambda|w|^2})$ is just the left-hand side of \eqref{Jessenineq}, we deduce:
$$
M_{\sigma,P_{n}}(e^{\lambda |w|^2})\leq K,
$$ 
where $K$ is independent of $n.$  Thus by Fatou lemma and Proposition \ref{M_Derivability} we conclude that
%Replacing  $e^{\lambda |w|^2}$ by $e^{\lambda |w|^2}\chi_{\calD_R}(w),$ where $\chi_{\calD_R}$ is the characteristic function of the closed disk $\calD_R,$ $R\to\infty,$ by the uniform convergence of $M_{\sigma,P_{n}}\to M_{\sigma}$ we obtain:
$$
\int_{\C}M_{\sigma}(w)e^{\lambda |w|^2}dw\leq K.
$$

Let us take $y$ such that $M_{\sigma,\calP_y}$ is a continuous function.
It is clear that if we remove all the terms corresponding to $p\leq y,$ and take $q>y$ large enough we obtain exactly the same bound for the function  $M_{\sigma}^{(y)}=\ast_{p\in \calP\setminus \calP_y}M_{\sigma,p}: $
$$
\int_{\C}M_{\sigma}^{(y)}(w)e^{\lambda |w|^2}dw\leq K.
$$
% Even the constant is the same though we do not need it

If $\calD_\rho=\overline{B(0, \rho)},$ $B=\overline{B(z,\rho)}$ denote the corresponding closed discs, $z\notin \calD_\rho,$ then 
$$
{e^{\lambda(|z|-\rho)^2}}\int_{B} M_{\sigma}^{(y)}(w)|dw|=\int_{B}e^{\lambda(|z|-\rho)^2}M_{\sigma}^{(y)}(w)|dw|\leq \int_{B}e^{\lambda|w|^2}M_{\sigma}^{(y)}(w)|dw|\leq K.
$$ 
% Since $0<|z|-\rho \leq |z+b|$ for any $b \in \calD_\rho.$

Let $\rho$ be large enough, so that $\calD_\rho$ contains the support of $M_{\sigma,\calP_{y}}.$ Then 
\begin{align*}
M_{\sigma}(z)&=(M_{\sigma,\calP_y}\ast M_{\sigma}^{(y)})(z)= \int_{\C}M_{\sigma,\calP_y} (w)M_{\sigma}^{(y)}(z-w)|dw|=\int_{\calD_\rho}M_{\sigma,\calP_y} (w)M_{\sigma}^{(y)}(z-w)|dw|\\
&\leq \sup_{\calD_\rho}M_{\sigma,\calP_y} (w)\cdot \int_{\C}M_{\sigma}^{(y)}(z-w)|dw|\leq Ke^{-\lambda(|z|-\rho)^2}\sup_{\calD_\rho}M_{\sigma,\calP_y} (w).
\end{align*}
As $y,\rho,\rho$ are independent of $z,$ we obtain that 
$$
M_{\sigma}(z)=O(e^{-\lambda|z|^2}).
$$

According to Corollary \ref{nonzero}, one can take $y$ large enough so that $M_{\sigma, \calP_y}$ has continuous partial derivatives of order up to $p$. We also have $D^{(a,b)}(f\ast g)=D^{(a,b)}(f)\ast g = f \ast D^{(a,b)}(g).$ Thus, the same arguments as above imply that the required estimate holds for partial derivatives of $M_\sigma(z)$ of any order $p.$
\end{proof}

%\begin{remark} The above estimate does not apply in general to $\tild{M}_s(z_1, z_2).$
%\end{remark}

\begin{coro}
\label{M_is_Scwartz}
The functions $M_\sigma(z)$ and $\tilde{M}_\sigma(z)$ belong to the Schwartz space, that is they go to zero as $|z|\to \infty$ faster than any inverse power of $|z|,$ as do all their derivatives.
\end{coro}

\begin{proof}
The statement is clear for $M_\sigma(z)$ by the above theorem. Now, $\tilde{M}_\sigma(z)=\calF M_\sigma(-z).$ Since $\calF$ maps Schwartz functions to Schwartz functions the result follows.
\end{proof}

\begin{coro}
$$
\tilde{M}_\sigma(z_1,z_2)=\int_{\C} M_\sigma(w)\psi_{z_1, z_2}(w) |dw|.
$$
\end{coro}

\begin{proof}
Each side of the above equality is an entire function of $z_1, z_2$ (the left one  by Proposition \ref{Euler_product}, the right one by Proposition \ref{M_Growth}). These functions are equal when $z_2=\bar{z}_1$ by Proposition \ref{M_Derivability}, thus they must coincide for any $z_1,z_2\in \C.$
\end{proof}

\begin{remark}
The last corollary also follows from Theorem \ref{maindistrib}, however we prefer to give a direct proof.
\end{remark}

\subsection{Proof of Theorem \ref{maindistrib}}

We will apply Lemma A from \cite[\S 5]{IM3}, which is a general result that allows to deduce from the convergence of averages for a special class of functions $\Phi,$ the same fact for more general $\Phi.$

First of all, Corollaries \ref{M_is_not_bad} and \ref{M_is_Scwartz} imply that $M_\sigma$ is a good density function on $\bbR^2$ in the sense of Ihara and Matsumoto, that is, it is non-negative, real valued, continuous, with integral over $\bbR^2$ equal to $1$, and such that both the function and its Fourier transform belong to $L^1\cap L^\infty.$

By \ref{corochar} the identity \eqref{Ldensity} holds for any additive character $\psi_z$ of $\C.$ Lemma A implies then that \eqref{Ldensity} is true for any bounded continuous $\Phi$, for the characteristic function of any compact subset of $\bbR^2$ or of the complement of such a subset.

Now, let us take $\phi_0(r)=\exp(ar).$ Proposition \ref{M_Growth} implies that $\int_{\C} M_\sigma(z)\phi_0(|z|)|dz|$ converges. The same reasoning as in \cite[\S 5.3, Sublemma]{IM3} allows us to see that $\Avg_{\chi\in\Gamma_m} \exp(a|\Lg(f\otimes\chi, s)|)\ll 1.$ This concludes the proof of Theorem \ref{maindistrib}.

\section{Average on primitive forms}
\label{section_primitive}

While working with modular forms it is analytically more natural to consider harmonic averages instead of usual ones. One introduces the harmonic weight 
$$
\omega(f)=\frac{\Gamma(k-1)}{(4\pi)^{k-1}(f,f)_N},
$$ 
where 
$$
(f,f)_N=\int_{\Gamma_0(N)\backslash\calH} |f(z)|^2y^k\frac{dx\,dy}{y^2}
$$ 
is the Petersson scalar product, $\calH=\{z\in \C\mid \Im z>0\}.$ We denote by $\Avgh G(f)$ the harmonic average
$$
\Avgh G(f)=\sum_{f\in B_k(N)} \omega(f)G(f).
$$

It can be proven \cite[Corollary 2.10 for $m=n=1$]{ILS} that for squarefree $N$ we have
\begin{equation}
\label{sumomega}
\sum_{f\in B_k(N)}\omega(f)=\frac{\varphi(N)}{N}+O\left(\frac{\tau(N)^2\log (2N)}{N k^{5/6}}\right),
\end{equation}
thus $\Avgh$ is an average operator when $\frac{\varphi(N)}{N} \to 1.$  

One has the following interpretation of $\omega(f)$ via the symmetric square $L$-functions \cite[Lemma 2.5]{ILS}:
\begin{equation}
\label{weightL}
\omega(f)=\frac{2\pi^2}{(k-1)N L(\Sym^2 f,1)}.
\end{equation}

% $N$ prime and $k$ fixed because of a very bad estimate for Petersson for $(m, N)\neq 1.$ One can make $k\to \infty,$ when $N=1.$ Can we make $N$ squarefree only?

\begin{theo} 
\label{average_primitive}
Assume that $N$ is a prime number and that $k$ is fixed.
Let $0<\epsilon<\frac{1}{2}$ and $T,R>0.$ Let  $s=\sigma+it$ belong to the domain $\sigma\geq\epsilon+\frac{1}{2},$ $|t|\leq T,$ and $z$ and $z'$ to a disc $\calD_R.$ Then, assuming GRH for $L(f, s),$ for any $\delta>0$ we have 
$$
\Avgh(\overline{\ggo(f,s,z)}{\ggo(f,s,z')})-\sum_{n, m\in\bbN}n^{-\bar{s}}m^{-s}\sum_{\substack{x\in J(n)\cap J(m)\\ (nm, N)=1}}\overline{c_{z,x}(n)}c_{z',x}(m)\ll_{\epsilon,R,T, \delta,k} N^{-\epsilon/2+\delta},
$$
and 
$$
\lim_{N\to +\infty}\Avgh(\overline{\ggo(f,s,z)}\ggo(f,s,z'))=\sum_{n,m\in\bbN}n^{-\bar{s}}m^{-s}\sum_{x\in J(n)\cap J(m)}\overline{c_{z,x}(n)}c_{z',x}(m).
$$

The convergence of the series is on the right-hand sides is uniform and absolute in the above domains without the assumption of GRH.
\end{theo}

\begin{remark}
In contrast to the situation, considered in Theorem \ref{averagechar}, we see that the average depends both on $\Re s$ and $\Im s.$ In fact, the independence of $\Im s$ in the case of averages with respect to characters is the corollary of the invariance of Haar measures on $\C^1$ under rotations.
\end{remark}

\begin{coro} Under the conditions of the previous theorem we have
$$
\lim_{N\to+\infty}\Avgh \psi_{z_{1},z_{2}}(\Lg(f,s))=\tilde{M}^h_{s}(z_{1},z_{2})=\sum_{n,m\in\bbN}n^{-\bar{s}}m^{-s}\sum_{x\in J(n)\cap J(m)}c_{z_1,x}(n)c_{z_2,x}(m).
$$
\end{coro}
\begin{proof}  We have 
\[
\psi_{z_{1},z_{2}}(\Lg(f,s))=\overline{\ggo(f,s,-\overline{z}_{1})} {\ggo(f,s,z_{2})},
\]
so
$$
\lim_{N\to+\infty}\Avgh \psi_{z_{1},z_{2}}(\Lg(f,s))=
\sum_{n,m\in\bbN}n^{-\bar{s}}m^{-s}\sum_{x\in J(n)\cap J(m)}\overline{c_{-\bar{z}_1,x}(n)}c_{z_2,x}(m).
$$
The corollary follows from the equality $\overline{c_{-\bar{z}, x}(n)}=c_{z, x}(n)$ which is implied by Lemma \ref{coefficients_conjugate}.
\end{proof}

\subsection{Naive approach}
\label{naive}

In this subsection we try to estimate the average in a naive way via Euler products. This approach works for $\Re s$ large enough and gives a formula which turns out to be valid for more general $s.$ The intermediate calculations will  be used again in \S \ref{avproof}. All the estimates are written assuming only that $N$ is squarefree and not assuming that $k$ is fixed until the very end of \S\ref{avproof}.

We have 
$$
\Avgh (\overline{\ggo(f,s,z)}\ggo(f,s,z'))= \sum_{f\in B_k(N)}\omega(f)\sum_{n,m\geq 1}n^{-\bar{s}}m^{-s}\overline{\lgo_{z}(n)}\lgo_{z'}(m).
$$

Let $\tau_k(n)=|\{(d_1,\dots,d_k)\in \bbN^k\mid d_1\cdot\dots\cdot d_k=n\}|$. We will use a version of the Petersson formula proven in \cite[Corollary 2.10]{ILS}. Note that our weights are slightly different from those used in \cite{ILS}, we follow instead \cite{RW} in our normalization.

\begin{prop}
\label{Peterson}
If $N$ is squarefree, $(m,N)=1,$ $(n,N^2)\mid N,$ then 
$$
S(m, n)=\sum_{f\in B_{k}(N)}\omega(f)\eta_{f}(m)\eta_{f}(n)=\frac{\varphi(N)}{N}\delta(m,n)+\Delta(m,n),
$$ 
where  $\delta(m,n)$ is the Kronecker symbol and
$$
\Delta(m,n) = O\left(k^{-\frac{5}{6}}(mn)^{\frac{1}{4}}N^{-1}(n, N)^{-1/2}\tau(N)^2\tau_3((m,n))\log(2mnN)\right),
$$
the implied constant being absolute.
\end{prop}

The conditions of this proposition are in particular satisfied once $(nm, N)=1.$ We will also use the following trivial bound, when $(m, N)\neq 1:$
\begin{equation}
\label{Petersonbad}
|S(m,n)| \leq  \sum_{f\in B_{k}(N)}\omega(f) \frac{\tau(m)\tau(n)}{\sqrt{(m, N)}} 
=\left(\frac{\varphi(N)}{N}+O\left(\frac{\tau(N)^2\log (2N)}{N k^{5/6}}\right)\right)\frac{\tau(m)\tau(n)}{\sqrt{(m, N)}},
\end{equation}
which holds by virtue of \eqref{sumomega} and the fact that $\ds|\eta_f(m)| \leq \frac{\tau(m)}{\sqrt{(m,N)}}$ since $N$ is squarefree. Obviously, the corresponding bound is also true if we assume instead that $(n, N)\neq 1.$

\begin{remark}
In what follows, one can possibly soften our restrictions on $N$ (in particular, remove the assumption that $N\to\infty$) by using more elaborate bounds on the sums in the case when $(mn, N)\neq 1$, applying directly the construction of an explicit basis of $S_k(N)$ from $B_k(N),$ in a way similar to \cite[Proposition 2.6]{ILS}.
\end{remark}

Using the above estimates, we can write  
\begin{align*}
\Avgh (\overline{\ggo(f,s,z)}&\ggo(f,s,z'))=\sum_{f\in B_k(N)}\omega(f)\sum_{n,m}n^{-\bar{s}}m^{-s}\overline{\lgo_{z}(n)}\lgo_{z'}(m)\\
&= \sum_{n,m}n^{-\bar{s}}m^{-s}\sum_{f\in B_k(N)}\omega(f)\overline{\lgo_{z}(n)}\lgo_{z'}(m)\\
&=  \sum_{n,m}n^{-\bar{s}}m^{-s}\sum_{x\in J_N(n), y\in J_N(m)}\overline{c^N_{z, x}(n)}c^N_{z',y}(m)\sum_{f\in B_k(N)}\omega(f)\eta_{f}(x)\eta_{f}(y)\\
&=  \sum_{n,m}n^{-\bar{s}}m^{-s}\sum_{\substack{x \in J_N(n),y\in J_N(m)\\(xy, N)=1}}\overline{c^N_{z, x}(n)}c^N_{z',y}(m)\left(\delta(x,y)\frac{\varphi(N)}{N}+\Delta(x,y)\right)\\ 
&+\sum_{n,m}n^{-\bar{s}}m^{-s}\sum_{\substack{x\in J_N(n), y\in J_N(m)\\(xy,N)\neq 1}}\overline{c^N_{z, x}(n)}c^N_{z',y}(m)\sum_{f\in B_k(N)}\omega(f)\eta_{f}(x)\eta_{f}(y)\\
&= \frac{\varphi(N)}{N}\sum_{n,m}n^{-\bar{s}}m^{-s}\sum_{\substack{x\in J_N(n)\cap J_N(m)\\ (x, N)=1}}\overline{c^N_{z,x}(n)}c^N_{z',x}(m)\\
&+ \sum_{n,m}n^{-\bar{s}}m^{-s} \sum_{\substack{x \in J_N(n),y\in J_N(m)\\(xy, N)=1}}\overline{c^N_{z,x}(n)}c^N_{z',y}(m)\Delta(x,y)\\
&+\sum_{n,m}n^{-\bar{s}}m^{-s}\sum_{\substack{x\in J_N(n), y\in J_N(m)\\(xy,N)\neq 1}}\overline{c^N_{z, x}(n)}c^N_{z',y}(m)S(x, y).
\end{align*}
The fact that the sum can be subdivided into three parts will be justified by the absolute convergence of the series for $\Re s$ large enough.

Put $\ds\tilde{M}(s)=\sum_{n,m}n^{-\bar{s}}m^{-s}\sum_{\substack{x\in J_N(n)\cap J_N(m)\\ (x, N)=1}}\overline{c^N_{z, x}(n)}c^N_{z',x}(m).$ Let us first note that the sum does not depend on $N,$ since $c^N_{z, x}(n)=c_{z, x}(n)$ if $(n, N)=1,$ and the coefficient $c^N_{z,x}(n)$ vanishes, once we have both $(x, N)=1,$ and $(n, N)\neq 1.$ This allow us to write $\ds\tilde{M}(s)=\sum_{n,m}n^{-\bar{s}}m^{-s}\sum_{\substack{x\in J(n)\cap J(m)\\ (nm, N)=1}}\overline{c_{z, x}(n)}c_{z',x}(m).$

Our goal is to verify that $\tilde{M}(s)$ gives the principal term of the asymptotic behaviour of $m(s, z, z').$  If $m\notin I(n),$ which is equivalent to $I(m)\neq I(n),$ the term $\overline{c_{z,x}(n)}c_{z', x}(m)$ vanishes. Therefore,
$$
\tilde{M}(s)=\sum_{\substack{n\in\bbN,\ (nm, N)=1\\ m\in I(n)}}n^{-\bar{s}}m^{-s}\sum_{\substack{x\in J(n)\cap J(m)}}\overline{c_{z,x}(n)}c_{z',x}(m).
$$

Let us define $r_{-}(n)$ to be the largest square dividing $n,$ and $r_{+}(n)$ to be the least square divisible by $n.$ So, if $n=p_{1}^{k_{1}}\dots p_{l}^{k_{l}},$ we have 
$$
n=p_{1}^{k_{1} \mod 2}\dots p_{l}^{k_{l} \mod 2} r_-(n)^2, \quad r_+(n)^2=p_{1}^{k_{1} \mod 2}\dots p_{l}^{k_{l} \mod 2} n,
$$ 
and the squarefree part of $n$ is equal to
$$
\ds p_{1}^{k_{1} \mod 2}\dots p_{l}^{k_{l} \mod 2}=\frac{r_+(n)}{r_-(n)}.
$$
% For  $n=p_{1}^{k_{1}}\dots p_{l}^{k_{l}},$ let $=p_{1}^{k_{1}[2]}\dots p_{l}^{k_{l}[2]}n$ the smallest square bigger than $n$ having the same prime factors as $n.$ An element of $I(p_{1}^{k_{1}}\dots p_{l}^{k_{l}})$ is of the form $p_{1}^{k_{1}[2]}\dots p_{l}^{k_{l}[2]}r^2,$ where $r$ is an integer. Taking absolute values, we obtain: 
Using this notation, we can write for $s=\sigma+it$
\begin{multline}
\label{rnotation}
\tilde{M}(s) = \sum_{n\geq 1}\sum_{r\geq 1} n^{-\sigma+it}\left(\frac {r_+(n)}{r_-(n)}\right)^{-\sigma-it} r^{-2s}\sum_{\substack{x\in J(n)\cap J(m)\\ (mn, N)=1}}c_{z,x}(n)c_{z',x}\left(\frac{r_+(n) r^2}{r_{-}(n)}\right)\\
= \sum_{n, r\geq 1} r_+(n)^{-2\sigma} r_-(n)^{2it} r^{-2s}\sum_{\substack{x\in J(n)\\ (mn, N)=1}} c_{z,x}(n)c_{z',x}\left(\frac{r_+(n) r^2}{r_{-}(n)}\right),
\end{multline}
so
\begin{equation*}
|\tilde{M}(s)|\leq \sum_{n,r\geq 1}(r_+(n))^{-2\sigma}r^{-2\sigma}\sum_{\substack{x\in J(n)\\ (mn, N)=1}}|c_{z,x}(n)|\left|c_{z',x}\left(\frac{r_+(n) r^2}{r_{-}(n)}\right)\right|.
\end{equation*}

There are $2^{\omega(n)}-1=2^l-1$ different $n$ giving the same $r_+(n).$ As $\omega(n)\ll \frac{\log n}{2+\log\log n}$ by \eqref{boundomega}, so $2^l\ll_\epsilon n^\epsilon,$ using Lemma \ref{MainEstimate} and \eqref{sizeJ} we see that the the sum $\tilde{M}(s)$ converges absolutely for $\Re s>1/2:$
$$
|\tilde{M}(s)|\ll_\epsilon \sum_{n, r \geq 1} n^\epsilon \cdot n^{-2\sigma} \cdot r^{-2\sigma} \cdot n^\epsilon \cdot n^\epsilon \cdot r^{-2\epsilon}\cdot n^\epsilon = \sum_{n\geq 1}n^{-2\sigma+4\epsilon} \sum_{r\geq 1} r^{-2\sigma + 2\epsilon}.
$$

Let us now see what happens with the error term. If we put 
$$
\Delta(s)=\sum_{n,m \geq 1}n^{-\bar{s}}m^{-s}\sum_{\substack{x\in J_N(n), y\in J_N(m)\\(xy,N)= 1}}\overline{c^N_{z,x}(n)}c^N_{z', y}(m) \Delta(x,y),
$$ 
from the Proposition \ref{Peterson} together with the estimate $\tau_3(n)\leq \tau(n)^3 \ll_\epsilon n^\epsilon$, and Lemma \ref{MainEstimate} we conclude that 
$$
|\Delta(s)|\ll_\epsilon \frac{\tau(N)^2\log N}{Nk^{5/6}} \sum_{m,n\geq 1} (mn)^{-\sigma+\frac{1}{4}+\epsilon}.
$$ 

In a similar way, putting 
$$
\Delta'(s)=\sum_{n,m \geq 1}n^{-\bar{s}}m^{-s}\sum_{\substack{x\in J_N(n), y\in J_N(m)\\(xy,N)\neq 1}}\overline{c^N_{z, x}(n)}c^N_{z',y}(m)S(x, y),
$$  
we get
$$
|\Delta'(s)|\ll_\epsilon \frac{1}{\sqrt{p_{\min}(N)}}\left(\frac{\varphi(N)}{N}+O\left(\frac{\tau(N)^2\log (2N)}{N k^{5/6}}\right)\right) \sum_{m,n\geq 1} (mn)^{-\sigma+\epsilon},
$$ 
where $p_{\min}(N)$ is the least prime factor of $N.$

These bounds only make sense for $\sigma=\Re s> 5/4,$ when the series converge. For these values of $s$ we conclude that the error terms tend to $0,$ once $p_{\min}(N)\to \infty$ (recall that we assume $N$ to be squarefree). In the next section we are going to show how the estimates can be pushed to the left from $\Re s > 5/4.$

\subsection{Integral representation}
\label{integralrep}

We introduce the following notation. Let $0<\epsilon'<\epsilon<\frac{1}{2},$ $s\in \C$ with $\sigma=\Re s \geq \frac{1}{2}+\epsilon,$ $c>\max(0,1-\sigma),$ $X\geq 1$ a parameter to be specified later. The symbol $\ll$ will depend on $\epsilon, R,$ and $T$ but this dependence will not be explicitly indicated. As before, we assume only that $N$ is squarefree (and not necessarily prime), and we do not suppose $k$ to be fixed. We will write $\ggo$ to denote $\ggo(f, s, z)$ when no ambiguity is possible. 

We use the techniques from \cite{IM3}, though it would be possible to employ the approximate functional equations instead, since they are available in our case. First, we establish the analogues of the propositions proven in \cite[\S2.2]{IM3}. 

\begin{lemma}
\begin{enumerate}
\item For $\ds\Re s \geq \frac{1}{2}+\epsilon$ we have $\ggo=\ggo_{+}- \ggo_{-},$ where the holomorphic functions $\ggo_+$ and $\ggo_-$ are defined by
$$
\ggo_{+}(f,s,z,X)=\frac{1}{2\pi i}\int_{\Re w=c}\Gamma(w)\ggo(f,s+w,z)X^w dw,
$$ 
and
$$
\ggo_{-}(f,s,z,X)=\frac{1}{2\pi i}\int_{\Re w=\epsilon'-\epsilon}\Gamma(w)\ggo(f,s+w,z)X^w dw.
$$ 
\item The function $\ggo_+$ has a Dirichlet series expansion
$$
\ggo_+=\sum_{n=1}^\infty\lgo_{z}(n) e^{-\frac{n}{X}}n^{-s}
$$
which is absolutely and uniformly convergent on compacts in $\C.$
\end{enumerate}
\end{lemma}
\begin{proof} The first statement admits exactly the same proof as the corresponding part of \cite[Proposition 2.2.1]{IM3} with Ihara and Matsumoto's property (A3) being replaced by Lemma \ref{estimateL} in our case.

As for the second statement, we have the Dirichlet series expansion 
$$
\ds\ggo(f,s,z)=\sum_{n=1}^\infty \lgo_z(n)n^{-s}.
$$ 
Taking in account that $\sigma+c>1,$ we see that
$$
\ggo(f,s+w,z)=\sum_{n=1}^\infty \lgo_z(n)n^{-s-w}
$$
is absolutely and uniformly convergent with respect to $\Im w$ on $\Re w=c.$ Exchanging the integration and summation and using
$$
\frac{1}{2\pi i} \int_{\Re w = c} \Gamma(w)a^{-w} dw = e^{-a},
$$
we obtain the desired expansion. The absolute and uniform convergence is clear for Lemma \ref{MainEstimate}.
\end{proof}

In what follows we will estimate $\ggo_{+}$ on average, which will give the main term, the function $\ggo_{-}$ will on the contrary be estimated individually for each $f.$ The following lemma bounds $\ggo_-$ in terms of the parameter $X$.

\begin{lemma}
\label{estimateminus}
Let $\Re s \geq 1/2+\epsilon.$ Then for any $f\in B_k(N),$ $0<\epsilon'<\epsilon,$ $T>0,$ for $|\Im(s)|\leq T$ we have
$$
|\ggo_{-}(f, s, z, X)|\ll_{\epsilon'}(NkX)^{\epsilon'}X^{-\epsilon}.  
$$
\end{lemma}
\begin{proof}
Once again our proof largely mimics that of \cite[Proposition 2.2.13]{IM3}. We need to estimate the integral 
$$
\ggo_{-}(f,s,z,X)=\frac{1}{2\pi i}\int_{\Re w =\epsilon'-\epsilon}\Gamma(w)\ggo(f,s+w,z)X^w dw.
$$ 
Clearly, $|X^w|=X^{\epsilon'-\epsilon}$ and it is well-known \cite[(2.2.9)]{IM3} that 
$$
\Gamma(w)\ll |\Im w|^{c-1/2}\exp\left(-\frac{\pi}{2}|\Im(w)|\right),
$$ 
when $|\Im w| \geq 1,$ $\Re w \leq c,$ so in our case $\Gamma(w)\ll \exp(-|\Im(w)|).$
Lemma \ref{estimateL} ensures that, putting $u=\Im(w)$ and $t=\Im(s),$ we have
$$
\log|\ggo(f,s+w,z)|\ll \ell(kN)^{1-2\epsilon'}\ell(t+u)^{1-2\epsilon'}.
$$
Therefore, there exists $C=C(T,\epsilon')$ such that 
$$
|\ggo(f,s+w)|\leq\exp \left(C\ell(Nk)^{1-2\epsilon'}(\log(|u|+1))^{1-2\epsilon'}\right)\leq\exp \left(C\ell(Nk)^{1-2\epsilon'}\log(|u|+1)\right).
$$ 
So, by comparison with the $\Gamma$-integral, we have
\begin{align*}
|\ggo_{-}(f,s+w,z, X)|&\ll X^{\epsilon'-\epsilon}\int_{0}^{+\infty}e^{-u}(u+1)^{C\ell(Nk)^{1-2\epsilon'}} du\\
&\ll X^{\epsilon'-\epsilon}\Gamma(C\ell(Nk)^{1-2\epsilon'}+1)\\
& \ll X^{\epsilon'-\epsilon}\exp(C\ell(Nk)^{1-2\epsilon'}\log (C\ell(Nk)^{1-2\epsilon'}))\\
& \ll X^{\epsilon'-\epsilon}\exp(C'\ell(Nk)^{1-2\epsilon'}\log (\ell(Nk))) \\
& \ll_{\epsilon'} X^{\epsilon'-\epsilon}\exp(\epsilon'\ell(Nk))\ll X^{\epsilon'-\epsilon} (Nk)^{\epsilon'},
\end{align*}
since for $Nk$ large enough depending on $T$ and $\epsilon',$ $C'\ell(Nk)^{-2\epsilon'}\log (\ell(Nk))<\epsilon'$ holds.
\end{proof}

\subsection{Averaging}
\label{avproof}

We now go back to averaging over primitive forms. We denote for simplicity $\ggo=\ggo(s, f, z),$ $\ggo'=\ggo(s,f,z')$ and we adopt similar notation for $\ggo_{\pm}$ and $\ggo'_{\pm}.$ 

First of all, using the decomposition established in \S \ref{integralrep}, we note that
\[
\Avgh(\overline{\ggo}\ggo')=\Avgh(\overline{\ggo}_+\ggo'_+)-\Avgh(\overline{\ggo}_+\ggo'_-)-\Avgh(\overline{\ggo}_-\ggo'_+)+\Avgh(\overline{\ggo}_-\ggo'_-).
\]

Our first goal is to prove that the average
\begin{align*}
\Avgh(\overline{\ggo}_+\ggo'_+)=\sum_{f\in B_k(N)}\omega(f)\sum_{n,m\geq 1}n^{-\bar{s}}m^{-s}\overline{\lgo_{z}(n)}\lgo_{z'}(m)e^{-\frac{n+m}{X}}.
\end{align*}
gives the main term of the asymptotic behaviour. The calculations of \S \ref{naive} allow us to decompose the above average as follows:
\begin{equation}
\label{decplus}
\Avgh(\overline{\ggo}_+\ggo'_+)=\tilde{M}(s, X)+\Delta(s, X)+\Delta'(s, X),
\end{equation}
with
\begin{align*}
\tilde{M}(s, X)&=\frac{\varphi(N)}{N}\sum_{n,m}n^{-\bar{s}}m^{-s}e^{-\frac{n+m}{X}}\sum_{\substack{x\in J_N(n)\cap J_N(m)\\ (x, N)=1}}\overline{c^N_{z,x}(n)}c^N_{z',x}(m),\\
\Delta (s,X)&=\sum_{n,m}n^{-\bar{s}}m^{-s}e^{-\frac{n+m}{X}} \sum_{\substack{x \in J_N(n),y\in J_N(m)\\(xy, N)=1}}\overline{c^N_{z,x}(n)}c^N_{z',y}(m)\Delta(x,y),\\
\Delta'(s, X)&=\sum_{n,m}n^{-\bar{s}}m^{-s}e^{-\frac{n+m}{X}}\sum_{\substack{x\in J_N(n), y\in J_N(m)\\(xy,N)\neq 1}}\overline{c^N_{z, x}(n)}c^N_{z',y}(m)S(x, y).
\end{align*}

Noting that $0<1-e^{-a}< \min(a,1)$ and fixing any $\alpha >0,$ we see that
\begin{align*}
|\tilde{M}(s)-\tilde{M}(s, X)|&=\frac{\varphi(N)}{N}\left|\sum_{n,m}n^{-\bar{s}}m^{-s}(1-e^{-\frac{n+m}{X}})\sum_{\substack{x\in J_N(n)\cap J_N(m)\\ (x, N)=1}}\overline{c^N_{z,x}(n)}c^N_{z',x}(m)\right|\\
&\leq \frac{\varphi(N)}{N}\sum_{\substack{n\leq \alpha X \\ m \leq \alpha X}}n^{-\sigma}m^{-\sigma}\frac{n+m}{X}\sum_{\substack{x\in J_N(n)\cap J_N(m)\\ (x, N)=1}}|c^N_{z,x}(n)||c^N_{z',x}(m)|\\ 
&+  \frac{\varphi(N)}{N}\sum_{\substack{n\geq \alpha X\\ \text{ or } \\ m \geq \alpha X}}n^{-\sigma}m^{-\sigma}(1-e^{-\frac{n+m}{X}})\sum_{\substack{x\in J_N(n)\cap J_N(m)\\ (x, N)=1}}|c^N_{z,x}(n)||c^N_{z',x}(m)|.
\end{align*}
The calculations of \S \ref{naive} together with the observation that (in the notation of \eqref{rnotation}) $r^2r_+(b)^2=mn$ result in the following bound valid for any $\epsilon''>0$:
\begin{multline*}
\sum_{\substack{n\geq \alpha X\\ \text{ or } \\ m \geq \alpha X}}n^{-\sigma}m^{-\sigma}(1-e^{-\frac{n+m}{X}})\sum_{\substack{x\in J_N(n)\cap J_N(m)\\ (x, N)=1}}|c^N_{z,x}(n)||c^N_{z',x}(m)|\\
\ll_{\epsilon''}\sum_{(rs)^2\geq \alpha X} (rs)^{-2\sigma+\epsilon''}\ll_{\epsilon''}\sum_{r\geq \sqrt{\alpha X}} r^{-2\sigma+\epsilon''}\ll_{\epsilon''} (\alpha X)^{1/2-\sigma+\epsilon''/2},
\end{multline*}
while the absolute convergence of the series for $\tilde{M}(s)$ implies 
\[
\sum_{\substack{n\leq \alpha X \\ m \leq \alpha X}}n^{-\sigma}m^{-\sigma}\frac{n+m}{X}\sum_{\substack{x\in J_N(n)\cap J_N(m)\\ (x, N)=1}}|c^N_{z,x}(n)||c^N_{z',x}(m)|\ll \alpha.
\]
Taking $\epsilon''$ small enough so that $\beta=1/2-\sigma+\epsilon''/2<0$ and $\alpha$ satisfying $\alpha=(\alpha X)^\beta,$ we finally see that 
\begin{equation}
\label{plusest1}
|\tilde{M}(s)-\tilde{M}(s, X)|\ll_{\epsilon''} \frac{\varphi(N)}{N} X^{\frac{1/2-\sigma+\epsilon''/2}{1/2+\sigma-\epsilon''/2}}\leq  \frac{\varphi(N)}{N} X^{\epsilon''/2-\epsilon}.
\end{equation}

Now, let us turn to the second and the third terms in \eqref{decplus}. Once again, applying the estimates from \S \ref{naive} we see that for any $\epsilon''>0$
\begin{gather*}
|\Delta(s, X)|\ll_{\epsilon''} \frac{\tau(N)^2\log N}{Nk^{5/6}} \sum_{m,n\geq 1} (mn)^{-\sigma+\frac{1}{4}+\epsilon''}e^{-\frac{m+n}{X}}, \\
|\Delta'(s, X)|\ll_{\epsilon''} \frac{1}{\sqrt{p_{\min}(N)}}\left(\frac{\varphi(N)}{N}+O\left(\frac{\tau(N)^2\log (2N)}{N k^{5/6}}\right)\right) \sum_{m,n\geq 1} (mn)^{-\sigma+\epsilon''}e^{-\frac{m+n}{X}}.
\end{gather*}
Bounding the sums via the corresponding improper integrals (cf. \cite[proof of Proposition 2.2.13]{IM3}), we get
\begin{gather}
\label{plusest2}
|\Delta(s, X)|\ll_{\epsilon''} \frac{\tau(N)^2\log N}{Nk^{5/6}} X^{3/2+2\epsilon''-2\epsilon},\\
\label{plusest3}
|\Delta'(s, X)|\ll_{\epsilon''}  \frac{1}{\sqrt{p_{\min}(N)}}\left(\frac{\varphi(N)}{N}+O\left(\frac{\tau(N)^2\log (2N)}{N k^{5/6}}\right)\right)  X^{1+2\epsilon''-2\epsilon}.
\end{gather}

In what follows, we will choose $X$ (as a function of $N$) in such a way that the right-hand sides in \eqref{plusest1}, \eqref{plusest2}, and \eqref{plusest3} tend to $0.$ With this choice of $X,$ taking $z=z'$ and using the absolute convergence of $\tilde{M}(s),$ we obtain
\[
\Avgh |\ggo_+|^2 \ll_{\epsilon''} 1, \qquad \Avgh |\ggo'_+|^2 \ll_{\epsilon''} 1
\] 

Let us estimate the remaining terms involving $\ggo_{-}$ and $\ggo'_-.$ By Lemma \ref{estimateminus} and \eqref{sumomega}
\begin{multline*}
\Avgh |\ggo_-|^2\ll_{\epsilon' , T} (NkX)^{2\epsilon'} X^{-2\epsilon}\sum_{f\in B_k(N)} \omega(f)\\
\leq \left(\frac{\varphi(N)}{N}+O\left(\frac{\tau(N)^2\log (2N)}{N k^{5/6}}\right)\right)(NkX)^{2\epsilon'} X^{-2\epsilon}.
\end{multline*}
We apply the Cauchy--Schwartz to get
\begin{multline}
\label{minusest1}
|\Avgh (\overline{\ggo}_+\ggo'_-)| + |\Avgh(\overline{\ggo}_-\ggo'_+)|+ |\Avgh (\overline{\ggo}_-\ggo'_-)|\\
\ll_{\epsilon'} \left(\frac{\varphi(N)}{N}+O\left(\frac{\tau(N)^2\log (2N)}{N k^{5/6}}\right)\right)(NkX)^{2\epsilon'} X^{-\epsilon},
\end{multline}
since $(NkX)^{2\epsilon'}\geq (NkX)^{\epsilon'}$ and $X^{-2\epsilon}\leq X^{-\epsilon}.$

Let us now turn to the case considered in the theorem, by assuming that $k$ is fixed and $N=p$ is prime. Assuming that $\epsilon'' < 2\epsilon,$ we have
\begin{gather*}
|\tilde{M}(s)-\tilde{M}(s, X)|\ll_{\epsilon''} X^{\epsilon''/2-\epsilon},\\
|\Delta(s, X)|\ll_{\epsilon''} \frac{\log p}{p} X^{3/2+2\epsilon''-2\epsilon},\\
|\Delta'(s, X)|\ll_{\epsilon''}  \frac{1}{\sqrt{p}} X^{1+2\epsilon''-2\epsilon},\\
|\Avgh (\overline{\ggo}_+\ggo'_-)| + |\Avgh(\overline{\ggo}_-\ggo'_+)|+ |\Avgh (\overline{\ggo}_-\ggo'_-)|\ll_{\epsilon',k} p^{2\epsilon'} X^{2\epsilon'-\epsilon}.
\end{gather*}

Taking $X=p^{1/2},$ we see that the above bounds lead to
\[
|\Avgh \overline{\ggo(f, s, z)}\ggo(f, s, z')- \tilde{M}(s)| \ll_\delta p^{-\epsilon/2+\delta},
\]
where $\delta$, which depends on $\epsilon'$ and $\epsilon'',$ can be taken arbitrarily small.

The second part of the theorem follows from the first.

\section{Open questions and remarks}
\label{section_question}

This section is devoted to a series of questions and remarks to complement the results of the paper. We hope to address at least some of them in subsequent articles. We start by the topics discussed in \S\ref{section_twists}.

\begin{question}
\label{quest_generality}
Can Theorem \ref{averagechar} be proven in a greater generality?
\end{question}

For example, one can consider $L$-functions of more general automorphic cusp forms and the average taken with respect to their twists by Hecke characters of imaginary quadratic number fields or algebraic function fields with a fixed place at infinity. As indicated in \cite{IM3}, going beyond imaginary quadratic number fields seems to be tricky since it involves essentially new problems related to the presence of non-trivial units. One can also consider averages over quadratic characters in the spirit of \cite{MM}

\begin{question}
What is an unconditional version of Theorem \ref{maindistrib}?
\end{question}

The unconditional results \cite[Theorem 1]{IM2} and \cite[Theorem 1.1]{IM4} suggest that it should be possible to prove similar statements in our case. 
%We hope to return to this question in the near future.

\begin{question}
Prove an analogue of Theorem \ref{maindistrib} for modular forms in the other situations within the framework of the cases (A), (B), (C) discussed in the introduction. 
\end{question}

Some results in this direction were established by Mastsumoto in \cite{M1} in the case (C), that is the equidistribution of $L(f, \sigma+it),$ when $\sigma$ is fixed and $t\in \bbR$ varies. It seems, however, that, even when considering averages of Dirichlet $L$-functions conditionally on GRH, this question has not been fully investigated, the most advanced results having been obtain only in the case (A).

\begin{question}
Carry out a more in-depth study of the functions $M$ and $\tilde{M}.$
\end{question}

In the case of Dirichlet characters this was done in \cite{Ih2}, \cite{Ih3}. One should be able to write down an explicit power series expansion of $\tilde{M}_s(z_1, z_2)$ in the variables $z_1, z_2$, establish its analytic continuation, study its growth, its zeroes, etc.

We next switch to the case of averages with respect to primitive forms of \S\ref{section_primitive}, where the results are far less complete.

\begin{question} Can one obtain Theorem \ref{average_primitive} with weaker assumptions on $N$? Can we let $k$ tend to infinity, while $N$ is fixed? Can we let $k+N\to \infty$? 
\end{question}

By following carefully the proof of Theorem \ref{average_primitive}, one can see that the limit statement is still true when $N=1$ and $k\to \infty$. Indeed, in this case $\Delta'$ is not present and the parameter $X$ cas be chosen to be equal to $k^{1/2}.$ This suggests that some  greater generality should be possible. The idea would be to use better bounds on averages of the Fourier coefficients of cusp forms with indices not coprime with $N,$ which should be possible by a careful treatement of an explicit basis of the space of old forms in the spirit of \cite{ILS}

\begin{question}
Prove an unconditional version of Theorem \ref{average_primitive}. 
\end{question}

Surprisingly enough, 
%in contrast to averages with respect to characters --- is it true?
a crude reasoning with Euler products does not seem to work even for $\Re s> 1.$ An unconditional version for $\Re s> 1/2$ will certainly be tricky to obtain even if one only considers characters $\psi_z$ as in\cite{IM2} and \cite{IM4}.

\begin{question}
Is it possible to establish value distribution results in the case harmonic averages over the set of primitive forms?
\end{question}

The reason we could not carry out the study analogous to that of \S \ref{section_distribution} is the absence of a local theory (at least in a straight-forward way). Indeed, the $\tilde{M}_s$ do not seem to admit an Euler product in this case. One could hope to rely on the interpretation of $\omega(f)$ via the symmetric square $L$-functions \eqref{weightL}, though there does not seem to be an easy way to do that.

\begin{question}
\label{quest_noweights}
Can one remove the harmonic weights in Theorem \ref{average_primitive}?
\end{question}

At least two approaches are available. The papers \cite{ILS}, \cite{KM1}, \cite{KM2} address a similar issue in different situations by using the interpretation \eqref{weightL} of the weights via $L(\Sym^2 f, 1).$ 

A more conceptual way would be to construct the local theory first. The results of Serre \cite{Se1} on the equidistribution of the eigenvalues of Hecke operators $T_p$ suggest that the local picture should be fairly clear. This would allow to establish the value distribution results missing in the case of harmonic averages. We plan to address this question in a forthcoming paper. 

\begin{question} Can one prove Theorem \ref{average_primitive} in greater generality for other types of automorphic forms? 
\end{question}

The first obvious step would be establishing it for $L(f\otimes\chi, s).$ For more general $L$-functions an appropriate trace formula would be necessary to replace Petersson.

\begin{question}
What is a function field version of Theorem \ref{average_primitive}?
\end{question}

The GRH being known in this case, unconditional results should not be very difficult to establish along the lines of this paper, once proper definitions are given.

\begin{question}
Establish the properties of $\tilde{M}$ functions in the case of averages with respect to primitive forms. 
\end{question}

Some peculiarities do arise compared to the case of characters. For example, $\tilde{M}_s(z_1, z_2)$ is no longer holomorphic in $s,$ since the average does depend on $s$ and $\bar{s}.$ The function is still entire in $z_1, z_2$ for fixed $s.$ Establishing its explicit power series expansion, analytic continuation, etc. seems to be of interest. The growth properties of $\tilde{M}$ seem to be much more delicate in our case, since they are proven using local results in the situation of Ihara and Matsumoto.

\begin{question}
Write an adelic version of Ihara's and Matsumoto's results, as well as of our results in the setting of modular forms.
\end{question}

This might shed some light on and give a better understanding of the functions $M,$ $\tilde{M},$ as well as of the relation of the global theory to the local one. One might also hope to be able to deal with the problems related to units in the number field case (c.f. Question \ref{quest_generality}).

\begin{question}
What are the arithmetic implications of our results?
\end{question}

The results of Ihara and Matsumoto give us a better understanding of the behaviour of the Euler--Kronecker constants of cyclotomic fields. More generally, since the $\log$ case of averaging results for $\Q$ concerns, in particular, zeta-functions of cyclotomic fields $\zeta_{\Q(\zeta_m)}(s),$ which are simply the products of $L(s, \chi)$ over primitive Dirichlet characters of conductors dividing $m,$ the results of Ihara and Matsumoto can be seen as a first step in the development of a finer version of the asymptotic theory of global fields from \cite{TV}, that gives non-trivial results for abelian extensions. This is not the case in \cite{TV}, since infinite global fields, containing infinite abelian subfields are asymptotically bad in the terminology of loc. cit. 

When one takes averages with respect to primitive forms, the results are close in spirit to the  asymptotic study of zeta-functions of modular curves $X_0(N),$ that can be written as 
$$
\zeta_{X_0(N)}(s)= \prod\limits_{f\in B_2(N)} L(f, s).
$$ 
Establishing a precise relation boils down to answering Question \ref{quest_noweights}. 

Note that even a cruder version of the asymptotic theory in the spirit of \cite{TV} has not been developed in this case. In the function field case this was to a significant extent done in \cite{Z}. A higher dimensional asymptotic theory in the characteristic zero case is yet to be constructed.

\end{document}